\documentclass[10pt,reqno,1p,sort&compress,longtitle]{elsarticle}

\usepackage{amsfonts,amssymb,amsthm}
\usepackage{enumitem}
\usepackage{mathtools}

\newtheorem{theorem}{Theorem}[section]
\newtheorem{lemma}[theorem]{Lemma}
\newtheorem{corollary}[theorem]{Corollary}
\newtheorem{proposition}[theorem]{Proposition}

\theoremstyle{definition}
\newtheorem{definition}[theorem]{Definition}

\newtheorem{remark}[theorem]{Remark}

\numberwithin{equation}{section}

\newcommand{\blangle}{\bigl\langle}
\newcommand{\brangle}{\bigr\rangle}

\newcommand{\ip}[1]{(#1)_\Omega}
\newcommand{\bip}[1]{\bigl(#1\bigr)_\Omega}

\newcommand{\aip}[1]{\left(#1\right)_\Omega}
\newcommand{\setof}[1]{[#1]}
\newcommand{\bsetof}[1]{\bigl[#1\bigr]}

\newcommand{\defn}{\coloneqq}
\newcommand{\nfed}{\eqqcolon}
\newcommand{\embeds}{\mathrel{\hookrightarrow}}
\newcommand{\embedsc}{\mathrel{{\mathrlap{\lhook\joinrel\relbar}{\,\,\hookrightarrow}}}}

\newcommand{\wto}{\mathrel{\rightharpoonup}}

\newcommand{\id}{\text{Id}}

\DeclareMathOperator{\diam}{diam}

\DeclareMathOperator{\Flux}{Flux}

\newcommand{\x}{{\mathrm{x}}}

\newcommand{\dd}{\,\mathrm{d}}
\newcommand{\eps}{\varepsilon}

\newcommand{\Le}{\mathrm{L}}
\newcommand{\LeLe}[2]{\Le_t^{#1}\Le_{\x}^{#2}}
\newcommand{\Lloc}{\Le_{\text{loc}}}

\renewcommand{\H}{\mathrm{H}}
\newcommand{\C}{\mathrm{C}}

\newcommand{\cA}{\mathcal{A}}

\newcommand{\cE}{\mathcal{E}}

\newcommand{\cI}{\mathcal{I}}

\newcommand{\cL}{\mathcal{L}}

\newcommand{\cO}{\mathcal{O}}

\newcommand{\cS}{\mathcal{S}}
\newcommand{\cT}{\mathcal{T}}
\newcommand{\cU}{\mathcal{U}}

\newcommand{\cY}{\mathcal{Y}}

\newcommand{\bN}{\mathbb{N}}
\newcommand{\bR}{\mathbb{R}}
\newcommand{\R}{\bR}
\newcommand{\N}{\bN}

\newcommand{\sj}{\mathsf{j}}

\newcommand{\sT}{\mathsf{T}}

\newcommand{\uad}{\cU_{\text{ad}}}

\newcommand{\optiprog}[4]{
     \begin{aligned}       
       #1 \quad & #2 \\
        \text{s.t.} \quad &\left\{ \quad
         \begin{aligned} #3 \end{aligned}\right.
      \end{aligned}       
      \tag{#4}
}

\usepackage{hyperref}

\begin{document}

\title{Optimal control of an
  energy-critical semilinear wave equation in 3D with spatially integrated control
  constraints}

\author[1,2]{Karl Kunisch} \ead{karl.kunisch@uni-graz.at}

\author[2]{Hannes Meinlschmidt\corref{cor1}} \ead{hannes.meinlschmidt@ricam.oeaw.ac.at}

\address[1]{Institute for Mathematics and Scientific Computing,
  University of Graz, Heinrichstra{\ss}e 36, 8010 Graz, Austria}

\address[2]{Johann Radon Institute for Computational and Applied
  Mathematics (RICAM), Altenberger Stra{\ss}e 69, 4040 Linz, Austria}

\cortext[cor1]{Corresponding author}



\begin{abstract} This paper is concerned with an optimal control problem subject to
  the $\H^1$-critical defocusing semilinear wave equation on a smooth and bounded
  domain in three spatial dimensions. Due to the criticality of the nonlinearity in
  the wave equation, unique solutions to the PDE obeying energy bounds are only
  obtained in special function spaces related to Strichartz estimates and the
  nonlinearity. The optimal control problem is complemented by pointwise-in-time
  constraints of Trust-Region type $\|u(t)\|_{\Le^2(\Omega)} \leq \omega(t)$. We
  prove existence of globally optimal solutions to the optimal control problem and
  give optimality conditions of both first- and second order necessary as well as
  second order sufficient type. A nonsmooth regularization term for the natural
  control space $\Le^1(0,\sT;\Le^2(\Omega))$, which also promotes sparsity in time of
  an optimal control, is used in the objective functional.
\end{abstract}

\maketitle

\section{Introduction} \label{sec:introduction}

We consider the optimal control problem
\begin{equation}
  \optiprog{\min_{y,u}}{\ell(y,u)}{u \in \uad, \\
    y~\text{is the solution to~\eqref{eq:main-equation},}}{OCP} \label{eq:OCP}
\end{equation}
where the underlying partial differential equation is the $\H^1$-critical defocusing
wave equation on a bounded domain $\Omega$ with smooth boundary in three spatial
dimensions over a finite interval $(0,\sT)$, complemented with homogeneous Dirichlet
boundary conditions, in the prototype form
\begin{equation}
  \label{eq:main-equation}\left.
    \begin{aligned}
      \partial_t^2 y - \Delta y + y^5 & = u && \text{in}~(0,\sT) \times \Omega, \\
      y & = 0 && \text{on}~(0,\sT) \times \partial\Omega, \\
      \bigl(y(0),\partial_t y(0)\bigr) & = (y_0,y_1) && \text{in}~\Omega.
    \end{aligned}
    \qquad\right\} \tag{CWE}
\end{equation}
We suppose that
$\xi_0 \defn (y_0,y_1) \in \H^1_0(\Omega) \times \Le^2(\Omega) \nfed \cE$ and
$u \in \Le^1(0,\sT;\Le^2(\Omega))$, which is the natural $\H^1$-setting for the wave
equation. The performance index $\ell$ for~\eqref{eq:main-equation} is chosen to be
\begin{multline*}
  \ell(y,u) 
  \defn \frac12\|y(\sT) - y_d\|_{\Le^2(\Omega)}^2 +
  \frac\gamma4\|y\|_{\Le^4(0,\sT;\Le^{12}(\Omega))}^4 \\
  + \beta_1 \|u\|_{\Le^1(0,\sT;\Le^2(\Omega))} + \frac{\beta_2}2
  \|u\|_{\Le^2(0,\sT;\Le^2(\Omega))}^2
\end{multline*}
for $y_d \in \Le^2(\Omega)$ and scaling parameters
$\gamma,\beta_1,\beta_2$. The objective in~\eqref{eq:OCP} is thus to
find a control $u \in \uad$ such that the associated solution
to~\eqref{eq:main-equation} $y(\sT)$ at time $\sT$ matches a given
profile $y_d$ as well as possible in the $\Le^2$-sense. The parameters
$\gamma, \beta_1$ and $\beta_2$ in $\ell$ are nonnegative. It will be
specifically mentioned if their positivity is required. While the
$\Le^2(0,\sT;\Le^2(\Omega))$ term describes a quadratic control cost, the
$\Le^1(0,\sT;\Le^2(\Omega))$ term is known to be sparsity
enhancing. The purpose of the $\Le^4(0,\sT;\Le^{12}(\Omega))$ norm
term will become clear later. The constraint set $\uad$ is of the form
\begin{equation*}
  \uad \defn \Bigl\{v \colon (0,\sT) \to \Le^2(\Omega) \colon
  \|v(t)\|_{\Le^2(\Omega)} \leq  \omega(t)~\text{f.a.a.}~t
  \in (0,\sT)\Bigr\}
\end{equation*}
for a measurable function $\omega$ which is nonnegative almost everywhere on
$(0,\sT)$.
It models a maximum overall input power at every time $t \in (0,\sT)$ for the
controls $u$. We emphasize that $\omega$ is not assumed to be bounded away from $0$
uniformly almost everywhere. It is thus possible to model e.g.\ a forced soft
``shut-off'' or decay to zero of $\|u(t)\|_{\Le^2(\Omega)}$ as $t \searrow T$ for
some $T \in (0,\sT]$.

\subsection*{Context}

The state equation~\eqref{eq:main-equation} is a semilinear wave equation.  Such
equations are of interest in several areas of natural sciences, in particular in
relation to mathematical physics~\cite{Lai14,Stu04,Stu04a,Hei52,Schi51,Joer59}, in
nonlinear elasticity~\cite{Lev88}, and the theory of vibrating strings~\cite{SHNF16}.

The exponent $5$ in the power-law nonlinearity $y^5$ in~\eqref{eq:main-equation} is
the $\H^1$-critical one since it satisfies $5 = \frac{n+2}{n-2}$, with $n = 3$ being
the space dimension. This terminology stems from the case $\Omega = \R^n$ where
``critical'' implies that (classical) solutions $y$
to~\eqref{eq:main-equation} with $u = 0$ are
invariant under the scaling
$\lambda \mapsto \lambda^{-\frac{n-2}{2}}y(\frac{t}{\lambda},\frac\x\lambda)$ and
thereby preserve $\|\xi_0\|_\cE$, cf.\ e.g.~\cite[Ch.~3.1]{T06}. A major difficulty
here is that one does not obtain a bound on $y \in \Le^5(0,\sT;\Le^{10}(\Omega))$ and
thus on $y^5 \in \Le^1(0,\sT;\Le^2(\Omega))$, which makes global-in-time existence
difficult to prove. It turns out that a uniform bound on
$y \in \Le^4(0,\sT;\Le^{12}(\Omega))$ is also sufficient due to
interpolation and $\Le^\infty(0,\sT;\Le^6(\Omega))$ energy conservation. The
nonlinearity is defocusing due to its sign, which does not play a role for local
existence of solutions to~\eqref{eq:main-equation}, but is crucial for long-time,
i.e., global-in-time, existence.

There is a large body of rather recent work about the analysis of
solutions to critical wave equations and their global-in-time
existence, including monographs (partially) dedicated to the topic
such as~\cite{T06} or~\cite{S08}. We focus on the works on space
dimension 3. Historically, global existence was proven first for
$\Omega = \R^3$ by Grillakis~\cite{Gri92} in~1992 for smooth solutions
and by Shatah and Struwe~\cite{SS94,SS93} for energy space solutions
in~1993/1994. Smith and Sogge were able to extend this result to the
exterior of convex obstacles~\cite{SS95}, and finally the case of a
bounded domain $\Omega$ with Dirichlet conditions was treated by Burq,
Lebeau, and Planchon~\cite{BLP08} in~2008. This was followed by the
treatment of the Neumann conditions case by Burq and
Planchon~\cite{BP09}. By now, improved Strichartz estimates compared
to the ones which were available in~\cite{BLP08} are proven by Blair,
Smith and Sogge~\cite{BSS09} and these allow for a slightly more
convenient existence proof as outlined in~\cite[Ch.~IV]{S08}. Of
course, this is all for the defocusing case since the mentioned works
establish global existence for initial data of arbitrary
size. Extensions of the mentioned results and especially the focusing
case are in the focus of current research; we mention
exemplarily~\cite{Don17,Ken17,JLSX16,DKM16,DKM16a}.

Semilinear wave equations have attracted significant interest in the
classical control theory community and are a subject of ongoing
research. Let us just mention for example~\cite{JL13,DL09,DLZ03,LD93},
where subcritical nonlinearities are considered, and the more recent
works~\cite{KSZ16,CFM16} for the critical case. On the other hand, the
literature regarding optimal control of semilinear wave equations and
especially the one about stronger nonlinearities appears to be rather
scarce. We refer to~\cite{RM04} and related works, where the focus
lies more on the state constraints imposed on the system, or
to~\cite{FRW96} with a mild nonlinearity. We are not aware of any work
related to the optimal control of a \emph{critical} semilinear wave
equation.

The contributions of this work are thus threefold:
\begin{itemize}[leftmargin=1.15em]
\item Up to now, the proofs of global existence of solutions to~\eqref{eq:main-equation}
  mentioned above do not incorporate forcing terms or controls $u$. We
  thus explicitly revisit the proof in~\cite{BLP08} to obtain a global existence
  result including the control. Due to the lack of a uniform bound on the
  nonlinearity, as mentioned above, the proof is quite sophisticated and taylored to
  the critical nature of the problem.
\item We show existence of globally optimal solutions to~\eqref{eq:OCP} and derive
  optimality conditions of both first-and second order type. Again, a particular
  point is that there is no uniform bound on the $\Le^1(0,\sT;\Le^2(\Omega))$ norm of
  $y^5$ for varying controls $u$ from the PDE~\eqref{eq:main-equation} which we thus
  have to enforce using the cost functional $\ell$ with $\gamma > 0$. The derivative
  of the associated term then enters the optimality conditions.
\item We consider the constraint set of pointwise-in-time Trust Region type $\uad$
  for the optimal control problem which seems not to have received much attention in
  the available literature so far. This plays a most prominent and demanding role in
  the derivation of second order necessary optimality conditions for~\eqref{eq:OCP}.
  It is of independent interest and a novel contribution in its own right.
\end{itemize}
Let us also point out that we could also consider more general nonlinearities $f(y)$
in place of $y^5$ in~\eqref{eq:main-equation}, as long as they exhibit comparable
growth and continuity properties. We have chosen to omit the technical details for
the sake of exposition.

\subsection*{Overview}
We first establish several basic but fundamental results about solutions
to~\eqref{eq:main-equation} in Section~\ref{sec:exist-uniq-local}. This includes the
concept of mild and weak solutions to~\eqref{eq:main-equation} and several important
estimates for solutions to wave equations. It ends with local-in-time wellposedness
of~\eqref{eq:main-equation}. As announced above, the solution regularity will then be
$(y,y') \in \C([0,T^\bullet);\cE)$ with
$y \in \Lloc^4([0,T^\bullet);\Le^{12}(\Omega))$ for some $T^\bullet \in (0,\sT]$ but
there will be no uniform bound in the latter space. In
Section~\ref{sec:global-solutions}, we establish that such local-in-time solutions in
fact exist globally in time in the energy space class by incorporating the
inhomogeneity $u$ into the related proof in~\cite{BLP08}. The optimal control problem
is treated in Section~\ref{sec:optimal-control}. After some preparatory
differentiability results, we prove existence of globally optimal controls
for~\eqref{eq:OCP} as well as necessary optimality conditions of first and second
order, and also second order sufficient conditions. For the latter, we need to assume
that $\beta_2 > 0$.

\subsection*{Notation and conventions}
We already mentioned above that we often consider the energy space
$\cE \defn \H^1_0(\Omega) \times \Le^2(\Omega)$ and throughout equip $\H^1_0(\Omega)$
with the norm $\|y\|_{\H^1_0(\Omega)} \defn \|\nabla y\|_{\Le^2(\Omega)}$. Moreover,
for a time-dependent function $y$, we write
\begin{equation*}
  \xi_y(t) \defn \bigl(y(t),\partial_t y(t)\bigr) \quad \text{and} \quad \bigl\|\xi_y(t)\bigr\|_{\cE}^2 \defn
  \bigl\|\nabla y(t)\bigr\|_{\Le^2(\Omega)}^2 + \bigl\|\partial_t y(t)\bigr\|_{\Le^2(\Omega)}^2.
\end{equation*}
We use $\ip{\cdot,\cdot}$ for the $\Le^2(\Omega)$ inner product and write
$A \lesssim B$ if there is a constant $C >0$ such that $A \leq C\cdot B$. If
necessary, a dependency of the constant $C$ on another quantity $D$ will be denoted
by $A \lesssim_D B$. All other notation will be standard. We consider all function
spaces to be real ones.

Lastly, the global time interval length $\sT$ is given and fixed for this work, but
we will sometimes have to deal with the theory of functions or solutions on intervals
other than $[0,\sT]$. In these cases, let $T \in (0,\sT]$, and consider, if
necessary, implicitly a dilation of $[t_0,t_1] \subseteq [0,\sT]$ to $[0,T]$ via
$T \defn t_1 -t_0$.

\section{Existence and uniqueness of local solutions} \label{sec:exist-uniq-local}

The classical notion of a solution in the energy space $\cE$ is that of a \emph{mild
  solution}. For this, let us introduce the Laplacian operator $\Delta$ in
$\Le^2(\Omega)$ given by
\begin{align*}
  D(\Delta) & \defn \Bigl\{\varphi \in \H^1_0(\Omega) \colon \exists f \in \Le^2(\Omega)
  \colon \bip{\nabla\varphi,\nabla\psi} = \bip{f,\psi}~\text{for all}~\psi \in
  \H^1_0(\Omega)  \Bigr\}, \\
  \Delta \varphi & \defn -f.
\end{align*}

\begin{definition}[Mild solution]
  \label{def:mild-solution}
  We say that the function $y$ is a \emph{mild solution} to~\eqref{eq:main-equation}
  on $[0,T]$ if $y \in \Le^4(0,T;\Le^{12}(\Omega))$ and $\xi_y \in \C([0,T];\cE)$
  with $\int_0^t \xi_y(s) \dd s \in D(\Delta) \times \H^1_0(\Omega)$ for all
  $t \in (0,T]$ and
  \begin{equation}\label{eq:def-mild-solution}
    \xi_y(t) =
    \xi_0
    +
    \begin{pmatrix}
      0 & \id \\ \Delta & 0
    \end{pmatrix} \int_0^t \xi_y(s) \dd s + \int_0^t
    \begin{pmatrix}
      0 \\ u(s) - y^5(s)
    \end{pmatrix}
    \dd s  \quad \text{for all}~t \in [0,T]
  \end{equation}
  is satisfied.
\end{definition}

\begin{remark}\label{rem:interpolation-strichartz}
  Let us briefly comment on the---at first glance---somewhat curious requirement
  $y \in \Le^4(0,T;\Le^{12}(\Omega))$ in the definition of a mild solution. By
  Sobolev embedding, we have $\H^1_0(\Omega) \embeds \Le^6(\Omega)$, so a mild
  solution $y$ can be considered as an element of
  $\Le^\infty(0,T;\Le^6(\Omega)) \cap \Le^4(0,T;\Le^{12}(\Omega))$. The H\"older
  inequality shows that
  \begin{equation}
    \label{eq:interpolation-l5l10}
    \|f\|_{\Le^5(0,T;\Le^{10}(\Omega))}^5 \leq \|f\|_{\Le^4(0,T;\Le^{12}(\Omega))}^4
    \|f\|_{\Le^\infty(0,T;\Le^6(\Omega))}
  \end{equation}
  and thus
  \begin{equation}
    \label{eq:interpolation-l1l2}
    \|f^5\|_{\Le^1(0,T;\Le^{2}(\Omega))} \leq  \|f\|_{\Le^4(0,T;\Le^{12}(\Omega))}^4
    \|f\|_{\Le^\infty(0,T;\Le^6(\Omega))}
  \end{equation}
  for all $f \in \Le^4(0,T;\Le^{12}(\Omega)) \cap
  \Le^\infty(0,T;\Le^6(\Omega))$. This shows that for a mild solution $y$ on $[0,T]$
  we have $y^5 \in \Le^1(0,T;\Le^2(\Omega))$ such that the second row
  in~\eqref{eq:def-mild-solution} is in fact self-consistent. Of course, we could
  have required the weaker condition $y \in \Le^5(0,T;\Le^{10}(\Omega))$ instead at
  this point; we will see the benefit of the stronger requirement later.
\end{remark}

Following~\cite{KSZ16}, we moreover introduce the following concept of solution for
the problem~\eqref{eq:main-equation} which is more suited to the optimal control
problem. It is named after the authors of~\cite{SS94,SS93}.

\begin{definition}[Shatah-Struwe solution]\label{def:weak-solution}
  We say that the function $y$ is a \emph{Shatah-Struwe solution}
  to~\eqref{eq:main-equation} on $[0,T]$ if $y \in \Le^4(0,T;\Le^{12}(\Omega))$ and
  $\xi_y \in \Le^\infty(0,T;\cE)$ with $\xi_y(0) = \xi_0$ satisfies the weak
  formulation
  \begin{multline*}
    -\int_0^T \bip{\partial_t y(t),\partial_t \varphi(t)} \dd t + \int_0^T
    \bip{\nabla y(t),\nabla\varphi(t)} \dd t + \int_0^T \bip{y^5(t),\varphi(t)}\dd t
    \\ = \int_0^T \bip{u(t),\varphi(t)} \dd t \qquad \text{for all}~\varphi \in
    \C_c^\infty\bigl((0,T) \times \Omega\bigr).
  \end{multline*}
\end{definition}

Note that the notion of a Shatah-Struwe solution in fact also makes sense if $y$ does
not have the additional integrability property. It is also meaningful for initial
data only from $\H^{-1}(\Omega) \times \Le^2(\Omega)$, since the solution will be
continuous with values in $\H^{-1}(\Omega) \times \Le^2(\Omega)$, and for
$u \in \Le^1(0,T;\H^{-1}(\Omega)$. For our purpose, we had supposed
$\xi_0 \in \cE = \H^1_0(\Omega) \times \Le^2(\Omega)$ and
$u \in \Le^1(0,T;\Le^2(\Omega))$; thus, we can in fact show that Shatah-Struwe and
mild solutions coincide and are unique. The additional integrability property
$y \in \Le^4(0,T;\Le^{12}(\Omega))$ is also important here. The first step is to
establish conservation of energy for Shatah-Struwe solutions, from which we
immediately obtain \emph{continuity}, and later also \emph{uniqueness} of such
solutions:

\begin{proposition}[Energy conservation,~{\cite[Prop.~3.3]{KSZ16}}]
  \label{prop:energy-conservation} Let $y$ be a Shatah-Struwe solution
  of~\eqref{eq:main-equation} on $[0,T]$. Then the \emph{energy function} $E_y$
  associated to $y$ given by
  \begin{equation*}
    E_y(t) \defn  \frac12\|\xi_y(t)\|_{\cE}^2 + \frac16 \|y(t)\|_{\Le^6(\Omega)}^6 - \int_0^t \bip{u(s),\partial_t y(s)} \dd s
  \end{equation*}
  is absolutely continuous. Moreover,
  $E_y(t) = E_y(0) = \frac12 \|\xi_0\|_{\cE}^2 + \frac16 \|y_0\|_{\Le^6(\Omega)}^6$
  for all $t \in [0,T]$, and we have $\xi_y \in \C([0,T];\cE)$.
\end{proposition}

Here, the continuity of $\xi_y$ follows from the continuity of $E_y$,
cf.~\cite[Ch.~3, Thm.~8.2]{LM72}.  The energy conservation gives us an \emph{a
  priori} bound on the $\C([0,T];\cE)$ norm of $\xi_y$ and on
$\|y\|_{\Le^\infty(0,T;\Le^6(\Omega))}$ which will prove very useful when proving
that solutions exist globally-in-time:

\begin{lemma}
  \label{lem:energy-conservation-estimate}
  Let $y$ be a Shatah-Struwe solution to~\eqref{eq:main-equation} on $[0,T]$.
  Then 
  we have
  \begin{equation}\label{eq:a-priori-estimate-energy}
    \|\xi_y\|_{\C([0,T];\cE)}^2 +  \|y\|_{\Le^\infty(0,T;\Le^6(\Omega))}^6
    \lesssim\|\xi_0\|_{\cE}^2 +  \|y_0\|_{\Le^6(\Omega)}^6 +
    \|u\|_{\Le^1(0,\sT;\Le^2(\Omega))}^2 \nfed E_0.
  \end{equation}
\end{lemma}

\begin{proof}
  From $E_y(t) = E_y(0)$, we clearly find
  \begin{multline*}
    \frac12\|\xi_y(t)\|_{\cE}^2 + \frac16 \|y(t)\|_{\Le^6(\Omega)}^6 - \|\partial_t
    y\|_{\Le^\infty(0,t;\Le^2(\Omega))}\|u\|_{\Le^1(0,t;\Le^2(\Omega))} \\\leq E_y(t)
    = E_y(0) = \frac12\|\xi_0\|_{\cE}^2 + \frac16 \|y_0\|_{\Le^6(\Omega)}^6,
  \end{multline*}
  so for every $\eps > 0$
  \begin{multline*}
    \sup_{t \in [0,T]} \left(\frac12\|\xi_y(t)\|_{\cE}^2 + \frac16
      \|y(t)\|_{\Le^6(\Omega)}^6\right) \\ \leq E_y(0) + \eps\|\partial_t
    y\|_{\Le^\infty(0,T;\Le^2(\Omega))}^2 + \frac1{4\eps}
    \|u\|_{\Le^1(0,T;\Le^2(\Omega))}^2.
  \end{multline*}
  This implies
  \begin{multline*}
    \frac12\|\xi_y\|_{\C([0,T];\cE)}^2 + \frac16
    \|y\|_{\Le^\infty(0,T;\Le^6(\Omega))}^6 \\\lesssim E_y(0) + \eps\|\partial_t
    y\|_{\Le^\infty(0,T;\Le^2(\Omega))}^2 + \frac1{4\eps}
    \|u\|_{\Le^1(0,T;\Le^2(\Omega))}^2,
  \end{multline*}
  and with $\eps$ small enough we can absorb the
  $\eps\|\partial_t y\|_{\Le^\infty(0,T;\Le^2(\Omega))}^2$ term in the left-hand side
  to obtain~\eqref{eq:a-priori-estimate-energy}.
\end{proof}

Next we show that Shatah-Struwe and mild solutions coincide.

\begin{lemma}
  \label{lem:weak-and-mild} A function $y$ is a mild solution
  to~\eqref{eq:main-equation} if and only if it is a Shatah-Struwe solution.
\end{lemma}

\begin{proof}
  Let $y$ be a mild solution to~\eqref{eq:main-equation}. Testing the second row
  in~\eqref{eq:def-mild-solution} with
  $\psi \otimes \zeta \in \C_c^\infty(0,T) \otimes \C_c^\infty(\Omega)$, we find
  \begin{multline*}
    \int_0^T \bip{\partial_t y(t),\partial_t \psi(t)\zeta} \dd t \\ = \bip{y_1,\zeta}
    \int_0^T \partial_t \psi(t) \dd t + \int_0^T\aip{\Delta\int_0^t y(s) \dd s,
      \zeta} \partial_t \psi(t) \dd t \\ + \int_0^T \aip{\int_0^t u(s)- y^5(s) \dd
      s,\zeta} \partial_t \psi(t) \dd t.
  \end{multline*}
  Using selfadjointness of the Laplacian on $\Le^2(\Omega)$ and Fubini's theorem, we
  continue with
  \begin{multline*}
    = \int_0^T\aip{\int_0^t y(s) \dd s , \Delta \zeta} \partial_t \psi(t) \dd t
    +\aip{ \int_0^T \int_s^T \bigl(u(s)- y^5(s) \bigr)
      \partial_t \psi(t) \dd t  ,\zeta}\dd s \\
    = \int_0^T \bip{y(s),- \Delta\zeta} \psi(t) \dd t- \int_0^T \bip{u(s)-
      y^5(s),\zeta} \psi(s) \dd s.
  \end{multline*}
  It remains to observe that
  $\bip{y(s),-\Delta\zeta} = \bip{\nabla y(s),\nabla \zeta}$, due to
  $\zeta \in \C_c^\infty(\Omega) \subset D(\Delta)$ and the definition of the
  Laplacian in $\Le^2(\Omega)$, and that
  $\C_c^\infty(0,T) \otimes \C_c^\infty(\Omega)$ is dense in
  $\C_c^\infty((0,T) \times \Omega)$~(\cite[Thm.~4.3.1]{F82}) to conclude that $y$ is
  a weak solution to~\eqref{eq:main-equation}.

  For the reverse assertion, let $y$ be a Shatah-Struwe solution
  to~\eqref{eq:main-equation}. We need to show that the first row
  in~\eqref{eq:def-mild-solution} is satisfied in
  $\H^1_0(\Omega)$. 
  Writing
  \begin{equation*}
    y(t) - y_0 = \int_0^t \partial_t y(s) \dd s,
  \end{equation*}
  at first in $\Le^2(\Omega)$, we observe that the left-hand side is in fact a
  continuous function in $\H^1_0(\Omega)$ due to
  Proposition~\ref{prop:energy-conservation}. This gives the assertion.

  For the second row and $\int_0^t y(s) \dd s \in D(\Delta)$, observe that from the
  definition of a weak solution, we have
  \begin{equation}\label{eq:weak-solution-second-deriv}
    \blangle \partial_t^2 y(t),\zeta \brangle_{\H^{-1}(\Omega),\H^1_0(\Omega)} = \bip{\nabla
      y(t),\nabla \zeta} + \bip{u(t)-y^5(t),\zeta} \quad \text{for all}~\zeta\in \C_c^\infty(\Omega)
  \end{equation}
  for almost all $t \in (0,T)$.  With $\partial_t y(t) \in \Le^2(\Omega)$ for almost
  all $t \in (0,T)$ together with $y_1 \in \Le^2(\Omega)$, we thus find
  \begin{align*}
    \bip{\partial_t y(t) - y_1,\zeta} & = \blangle \partial_t y(t) - y_1,\zeta
    \brangle_{\H^{-1}(\Omega),\H^1_0(\Omega)} \\ & = \int_0^t \blangle \partial_t^2 y(s),\zeta \brangle_{\H^{-1}(\Omega),\H^1_0(\Omega)} \dd
    s \\ & =  \int_0^t \bip{\nabla
      y(s),\nabla \zeta} + \bip{u(s)-y^5(s),\zeta} \dd s
  \end{align*}
  and so
  \begin{equation*}
    \bip{\partial_t y(t) - y_1 + \int_0^t y^5(s) -  u(s) \dd s,\zeta} = \aip{\int_0^t y(s) \dd s,\Delta \zeta}.
  \end{equation*}
  Since this equality extends from all $\zeta$ in $\C_c^\infty(\Omega)$ to all
  $\zeta \in D(\Delta)$, we have by definition of the adjoint operator and
  selfadjointness of $\Delta$:
  \begin{equation*}
    \int_0^t y(s) \dd s  \in D(\Delta) \quad \text{and} \quad \Delta\int_0^t
    y(s) \dd s = \partial_t y(t) - y_1 + \int_0^t y^5(s) -  u(s) \dd s.
  \end{equation*}
  This is the second row in~\eqref{eq:def-mild-solution}. Continuity of the
  Shatah-Struwe solution follows from the energy inequality as noted in
  Proposition~\ref{prop:energy-conservation}.\end{proof}

Another consequence of the energy conservation for Shatah-Struwe and, per
Proposition~\ref{prop:energy-conservation}, mild solutions, is \emph{uniqueness}:

\begin{proposition}[Shatah-Struwe uniqueness,~{\cite[Cor.~3.4]{KSZ16}}]
  The Shatah-Struwe solution of~\eqref{eq:main-equation} is unique, if it
  exists.  \label{prop:shatah-struwe-uniqueness}
\end{proposition}

\begin{corollary}
  \label{cor:mild-l4l12-unique} If there exists a mild solution $y$
  to~\eqref{eq:main-equation}, then it is unique and coincides with the Shatah-Struwe
  solution.
\end{corollary}

Next, we establish some important estimates and finally local-in-time existence of
solutions. For this purpose, consider the block operator in
Definition~\ref{def:mild-solution} (mild solution) as a closed operator in
$\H^1_0(\Omega) \times \Le^2(\Omega)$ via
\begin{equation*}
  \cA \defn
  \begin{pmatrix}
    0 & \id \\ \Delta & 0
  \end{pmatrix}, \quad D(\cA) \defn D(\Delta) \times \H^1_0(\Omega).
\end{equation*}
It can be shown that the operator $\cA$ generates a $C_0$-semigroup
$t \mapsto e^{\cA t}$ on $\cE$, cf.~\cite[Ch.~XVII~\S3 Sect.~3.4]{DL00}. Thus, for
$\xi_{0,z} = (z_1,z_0) \in \cE$ and $f \in \Le^1(0,T;\Le^2(\Omega))$, the usual
\emph{variation-of-constants
  formula} \begin{equation}\label{eq:variation-of-constants-classical} \xi_z(t) =
  e^{\cA t}\xi_{0,z} + \int_0^t e^{\cA (t-s)} \begin{pmatrix} 0 \\ f(s)
  \end{pmatrix}
  \dd s.
\end{equation} is well defined and gives the unique mild
solution $z \in \C([0,T];\cE)$ to the linear wave equation
\begin{equation}\left.
    \begin{aligned}
      \partial_t^2 z - \Delta z  & = f && \text{in}~(0,T) \times \Omega, \\
      z & = 0 && \text{on}~(0,T) \times \partial\Omega, \\
      \bigl(z(0),\partial_t z(0)\bigr) & = (z_0,z_1) && \text{in}~\Omega,
    \end{aligned}\qquad \right\}\tag{LWE}\label{eq:LWE}
\end{equation}
cf.~\cite[Prop.~3.1.16]{ABNH11}.

We state the fundamental estimates for $z$
satisfying~\eqref{eq:variation-of-constants-classical}. While the first one is the
standard energy space estimate which follows immediately
from~\eqref{eq:variation-of-constants-classical} and Sobolev embedding, the second
one is a nontrivial Strichartz estimate as proven in~\cite{BSS09}:

\begin{lemma}[Energy- and Strichartz estimates] \label{lem:estimates} Let
  $f \in \Le^1(0,T;\Le^2(\Omega))$ as well as $\xi_{0,z} \in \cE$, and let $\xi_z$ be
  given by~\eqref{eq:variation-of-constants-classical} for $t \in [0,T]$. Then there
  exist constants $C_e(T),C_s(T)$ such that the \emph{energy estimate}
  \begin{equation}\label{eq:energy-estimate}
    \bigl\|\xi_z(t)\bigr\|_{\cE} + \|y(t)\|_{\Le^6(\Omega)} \leq C_e(T)\Bigl(\bigl\|\xi_{0,z}\bigr\|_{\cE} + \bigl\|f\bigr\|_{\Le^1(0,T;\Le^2(\Omega))}\Bigr)
  \end{equation}
  for all $t \in [0,T]$, and the \emph{Strichartz estimate}
  \begin{equation}\label{eq:strichartz-estimate}
    \bigl\|z\bigr\|_{\Le^4(0,T;\Le^{12}(\Omega))} \leq C_s(T) \Bigl(\bigl\|\xi_{0,z}\bigr\|_{\cE} + \bigl\|f\bigr\|_{\Le^1(0,T;\Le^2(\Omega))}\Bigr)
  \end{equation}
  are satisfied.
\end{lemma}

\begin{remark}
  \label{rem:norm-solution-operator-monotone}
  Suppose that $\xi_{0,z} = 0$. Then it is easy to see that the constants $C_e(\tau)$
  and $C_s(\tau)$ associated to the estimates~\eqref{eq:energy-estimate}
  and~\eqref{eq:strichartz-estimate} for solutions on the intervals $[0,\tau]$ are
  monotonously increasing in $\tau$. In other words, given a function
  $f \in \Le^1(0,T;\Le^2(\Omega))$, we have $C_e(\tau) \leq C_e(T)$ and
  $C_s(\tau) \leq C_s(T)$ for all $\tau \in [0,T]$.
\end{remark}

Now, if $\xi_y$ with $y \in \Le^4(0,T;\Le^{12}(\Omega))$ is given
by~\eqref{eq:variation-of-constants-classical} with $f = u - y^5$ and initial data
$\xi_0$, so that
\begin{equation}\label{eq:variation-of-constants}
  \xi_y(t) = e^{\cA t}\xi_0 + \int_0^t e^{\cA (t-s)} \begin{pmatrix}
    0 \\ u(s) - y^5(s)
  \end{pmatrix}
  \dd s,
\end{equation}
then $y$ is in fact the unique mild solution to~\eqref{eq:main-equation} in the sense
of Definition~\ref{def:mild-solution} on $[0,T]$ (see
again~\cite[Prop.~3.1.16]{ABNH11}) and thus also the Shatah-Struwe solution, recall
Corollary~\ref{cor:mild-l4l12-unique}. In particular, $y$ satisfies the estimates in
Lemma~\ref{lem:estimates} for $f = u - y^5$ and initial data $\xi_0$. As explained in
Remark~\ref{rem:interpolation-strichartz}, the integrability
$y^5 \in \Le^1(0,T;\Le^2(\Omega))$ follows from the additional dispersion information
$y \in \Le^4(0,T;\Le^{12}(\Omega))$.

Using the linear estimates in Lemma~\ref{lem:estimates} together with the
interpolation inequalities as in Remark~\ref{rem:interpolation-strichartz}, local
existence and uniqueness of a function $y$
satisfying~\eqref{eq:variation-of-constants}---which is then also the unique mild and
Shatah-Struwe solution on the interval of existence---depending continuously on the
given data follows from a standard fixed point argument. (See
also~\cite[Prop.~3.1]{KSZ16} for another explicit proof.)

\begin{theorem}[Local-in-time existence]
  \label{thm:local-existence} There exists a maximal unique mild and Shatah-Struwe
  solution to~\eqref{eq:main-equation} such that~\eqref{eq:main-equation} is well
  posed with respect to $\cE$ and $\Le^1(0,\sT;\Le^2(\Omega))$. More precisely, there
  exists a maximal time $T^\bullet \in (0,\sT]$ depending on $\xi_0 \in \cE$ and
  $u \in \Le^1(0,\sT;\Le^2(\Omega))$ and a unique function $y$ on $[0,T^\bullet)$
  with $\xi_y$ given by~\eqref{eq:variation-of-constants} such that
  \begin{equation*}
    \xi_y \in \C\bigl([0,T^\bullet);\cE\bigr) \quad \text{and} \quad y \in \Lloc^4\bigl([0,T^\bullet);\Le^{12}(\Omega)\bigr).
  \end{equation*}
  This function $y$ is the unique mild and Shatah-Struwe solution
  to~\eqref{eq:main-equation} on $[0,T]$ for every $T \in (0,T^\bullet)$. Moreover,
  there exists $\eps > 0$ such that for every initial value
  $\zeta_0 \in B_\eps(\xi_0) \subset \cE$ and every right-hand side
  $v \in B_\eps(u) \subset \Le^1(0,\sT;\Le^2(\Omega))$ there exists a unique solution
  $\bar y$ in the foregoing sense on the same intervals of existence and the mapping
  $B_\eps(\xi_0) \times B_\eps(u) \ni (\zeta_0,v) \to \bar y$ is continuous.
\end{theorem}

\begin{remark}
  \label{rem:wellposed} The wellposedness of~\eqref{eq:main-equation} allows to
  obtain certain classical properties of more regular solutions
  to~\eqref{eq:main-equation} and related equations also for Shatah-Struwe solutions
  by approximation. This includes in particular \emph{finite speed of propagation} of
  such solutions, cf.\ e.g.~\cite[Ch.~2.4.3, Thm.~6]{E98}.
\end{remark}

\section{Global solutions} \label{sec:global-solutions}


In this section, we establish that there indeed exists a unique \emph{global-in-time}
mild and Shatah-Struwe solution to~\eqref{eq:main-equation} on $[0,\sT]$ given
by~\eqref{eq:variation-of-constants}. The argument follows~\cite{BLP08} which in turn
builds upon~\cite{SS95}, see also~\cite[Ch.~IV, \S3]{S08} or~\cite[Ch.~5.1]{T06}. We
mostly outline the strategy and give a minimally invasive modification of the proof
in~\cite{BLP08}. The modification is necessary in the first place because on the one
hand, only the case $u = 0$ is treated in~\cite{BLP08}, and on the other hand, the
improved Strichartz estimates in~\cite{BSS09} as stated in Lemma~\ref{lem:estimates}
allow to simplify the proof at some places compared to~\cite{BLP08}. The global
existence result is also stated in~\cite[Thm.~3.8]{KSZ16}, however, without a proof.

The proof consists in principle of an ordinary extension argument via the energy
space $\cE$: Let $y$ be the maximal solution to~\eqref{eq:main-equation} as in
Theorem~\ref{thm:local-existence}. We show that the limit
$\lim_{t \nearrow T^\bullet} \xi_y(t) \nfed \xi_y(T^\bullet)$ exists in $\cE$. Then
we either have already $T^\bullet = \sT$, or we can extend the solution by
re-applying Theorem~\ref{thm:local-existence} starting from $T^\bullet$ with initial
data $\xi_y(T^\bullet)$ until we have a solution on the whole $[0,\sT]$.

For this purpose, it is imperative to observe that, by the energy
estimate~\eqref{eq:energy-estimate}, the limit of $\xi_y(t)$ as
$t \nearrow T^\bullet$ in~\eqref{eq:variation-of-constants} only fails to exist if
$y \notin \Le^5(0,T^\bullet;\Le^{10}(\Omega))$; so we want to prove that
in fact $y \in \Le^5(0,T^\bullet;\Le^{10}(\Omega))$. Due to energy conservation (cf.\
Lemma~\ref{lem:energy-conservation-estimate}) and~\eqref{eq:interpolation-l5l10}, it
is moreover sufficient to show that $y \in \Le^4(0,T^\bullet;\Le^{12}(\Omega))$.

In the introduction it was mentioned several times that there is no direct bound on a
local solution $y$ in $\Le^4(0,T^\bullet;\Le^{12}(\Omega))$ in the critical
case. This is in contrast to the subcritical case with a nonlinearity $y^p$ with
$1 < p < 5$. We give a quick demonstration of this and how
it does not work for the critical case (cf.~\cite[Ch.~IV.2]{S08}). Let $y$ be the local-in-time solution
of~\eqref{eq:main-equation} as in Theorem~\ref{thm:local-existence} and let
$t \in [0,T^\bullet)$. The H\"older inequality yields the more general form
of~\eqref{eq:interpolation-l1l2}
\begin{equation*}
  \bigl\|y^p\bigr\|_{\Le^1(t,T^\bullet;\Le^2(\Omega))} \leq
 \|y\|_{\Le^4(t,T^\bullet;\Le^{12}(\Omega))}^{p-1}
  \|y\|_{\Le^{\frac{4}{5-p}}(t,T^\bullet;\Le^{\frac{12}{7-p}}(\Omega))}. 
\end{equation*}
Now note that $\frac{12}{7-p} < p+1$ if and only if $1 < p < 5$. Hence
\begin{equation*}
  \|y\|_{\Le^{\frac{4}{5-p}}(t,T^\bullet;\Le^{\frac{12}{7-p}}(\Omega))}
  \lesssim_\Omega
  \bigl(T^\bullet-t\bigr)^{\frac{5-p}{4}}\|y\|_{\Le^\infty(t,T^\bullet;\Le^{p+1}(\Omega))}.
\end{equation*}
If $y^p$ is the nonlinearity in~\eqref{eq:main-equation}, then
$\|y\|_{\Le^\infty(t,T^\bullet;\Le^{p+1}(\Omega))}$ is present
in Proposition~\ref{prop:energy-conservation}
instead of the corresponding $\Le^6(\Omega)$ terms. Accordingly, it is uniformly
bounded by a power of $E_0$ and we obtain from the above
\begin{equation*}
  \bigl\|y^p\bigr\|_{\Le^1(t,T^\bullet;\Le^2(\Omega))} \lesssim_{\Omega,E_0}
  \bigl(T^\bullet-t\bigr)^{\frac{5-p}{4}}\|y\|_{\Le^4(t,T^\bullet;\Le^{12}(\Omega))}^{p-1}.
\end{equation*}
The Strichartz estimate~\eqref{eq:strichartz-estimate} then shows that
\begin{equation*}
  \|y\|_{\Le^4(t,T^\bullet;\Le^{12}(\Omega))} \lesssim_{\Omega,E_0} \sqrt{E_0}
  + \bigl(T^\bullet-t\bigr)^{\frac{5-p}{4}}\|y\|_{\Le^4(t,T^\bullet;\Le^{12}(\Omega))}^{p-1}.
\end{equation*}
Via Lemma~\ref{lem:useful-estimate}, this implies that there is $t^\star \in (0,T^\bullet)$ such that
$\|y\|_{\Le^4(t^\star,T^\bullet;\Le^{12}(\Omega))}$ is finite. Unfortunately, the
foregoing proof breaks down completely for the critical value $p=5$ since we obtain
\begin{equation*}
  \|y\|_{\Le^4(t,T^\bullet;\Le^{12}(\Omega))} \lesssim_{\Omega} \sqrt{E_0}
  + \|y\|_{\Le^\infty(t,T^\bullet;\Le^6(\Omega))}\|y\|_{\Le^4(t,T^\bullet;\Le^{12}(\Omega))}^{4}
\end{equation*}
and $\|y\|_{\Le^\infty(t,T^\bullet;\Le^6(\Omega))}$, although bounded by $E_0^{1/6}$,
does not go to zero as $t \searrow T^\bullet$ in general. It is thus necessary to
proceed differently. The idea is to replace the ``full-spacetime'' norm
$\|y\|_{\Le^\infty(t,T^\bullet;\Le^6(\Omega))}$ by a localized one which indeed goes
to zero as $t \searrow T^\bullet$. This is done as follows.

We first show an $\Le^6$-\emph{non-concentration effect} in $T^\bullet$, namely that
the $\Le^6$ norm of the solution cannot concentrate in a single point $\x_0$, i.e.,
be greater than $0$. This is the most involved and nontrivial result, but luckily we
only need to make appropriate modifications to incorporate the inhomogeneity $u$
compared to the proof in~\cite[Prop.~3.3]{BLP08}; see
Proposition~\ref{prop:l6-nonconcentration}. The non-concentration effect allows to
prove that the solution must be in $\LeLe4{12}$-integrable on a backwards light cone
through $(T^\bullet,\x_0)$. (For precise definitions, see below.)  This is done in
Proposition~\ref{prop:l-infty-disc-implies-l4l12-cone}. We then show in
Proposition~\ref{prop:l5-l10-small-then-cone-bigger} that this
$\LeLe4{12}$-integrability enables us to prove that the $\LeLe\infty6$-norm of $y$
becomes arbitrarily small on a \emph{slightly larger} light cone as we approach
$T^\bullet$. This allows to employ an argument similar to the one displayed for the
subcritical case above which then finally leads to boundedness of
$\Le^4(t^\star,T^\bullet;\Le^{12}(\Omega))$ for some $t^\star$ close to $T^\bullet$
and thus finishes the proof of the main result, Theorem~\ref{thm:global-existence}.

\subsection{Global existence}

We fix $\x_0 \in \overline\Omega$ for the following, if not stated otherwise, as well
as the blowup time $T^\bullet >0$. Frequently needed objects are the
$\delta$-enlarged \emph{backwards light cone} through $(s,\x_0)$ for
$0 \leq t_0 \leq s \leq T^\bullet$ and $\delta \geq 0$ given by
\begin{equation*}
  \Lambda(\delta;t_0,s) \defn \Bigl\{(t,\x) \in [t_0,s] \times \overline\Omega \colon |\x-\x_0| \leq
  \delta + T^\bullet - t\Bigr\}.
\end{equation*}
Moreover, we need its ``time slice'' at time level $\tau$
\begin{equation*}
  D_\tau^\delta \defn \Bigl\{ (\tau,\x) \colon \x \in \overline\Omega,~|\x-\x_0| \leq
  \delta + T^\bullet - \tau\Bigr\}
\end{equation*}
with $D_\tau \defn D_\tau^0$. Clearly, the sets $D_\tau^\delta$ live in the
four-dimensional space $\R\times\R^3$, and we use $P_\x D_\tau^\delta$ to denote the
projection of this set onto the second coordinate block, so in $\R^3$. We moreover
use the mixed Lebesgue norm notation
\begin{equation*}
  \|w\|_{\LeLe{p}{q}(\Lambda(\delta;t_0,s))}^p \defn \int_{t_0}^s \left(
    \int_{P_\x D_t^\delta}
    |w(t,\x)|^q \dd \x\right)^{p/q} \dd t,
\end{equation*}
with the usual modification for $p = \infty$. Finally, let us define the local energy
of a function $v$ on $D_t^\delta$ for $0 \leq t < T^\bullet$ by
\begin{equation*}
  E_v(\delta;t) \defn \int_{P_\x D_t^\delta} \frac{|\nabla v(t,\x)|^2 +
    |\partial_t v(t,\x)|^2}2 + \frac{|v(t,\x)|^6}6 \dd\x.
\end{equation*}

The first result is the $\Le^6$-non-concentration effect:

\begin{proposition}[$\Le^6$-nonconcentration]
  \label{prop:l6-nonconcentration}
  There holds
  \begin{equation*}
    \lim_{t\nearrow T^\bullet} \int_{P_\x D_t} |y(t,\x)|^6 \dd \x = 0.
  \end{equation*}
\end{proposition}

\begin{proof}
  We need only make appropriate modifications in the proof
  in~\cite[Prop.~3.3]{BLP08}, whose strategy follows~\cite[Lem.~3.3]{SS95} or
  \cite[Ch.~V, Prop.~3.2]{S08}, to incorporate the inhomogeneity $u$. There are
  essentially three aspects:

  \begin{enumerate}[wide, labelwidth=!, labelindent=0pt]
  \item The estimate
    \begin{equation*}
      \|\partial_\nu y\|_{\Le^2(0,s;\Le^2(\partial\Omega))}^2 \lesssim E_0
    \end{equation*}
    is still satisfied uniformly for every $0 \leq s < T^\bullet$, where
    $\partial_\nu$ is the trace of the outer unit normal on $\partial\Omega$. This
    follows as in~\cite[Prop.~3.2]{BLP08} by taking care of the estimate
    \begin{equation*}
      \int_0^{s} \int_\Omega \Bigl[(Zu)(t,\x)\cdot y(t,\x) - (Zy)(t,\x)\cdot u(t,\x)
      \Bigr] \dd \x \dd t
      \lesssim E_0
    \end{equation*}
    uniformly in $s$, where $Z$ is a smooth scalar field on $\Omega$ which coincides
    with $\partial_\nu$ on $\partial\Omega$. (For this argument, we suppose $u$ and
    $y$ to be smooth and refer to the wellposedness of the equation, recall
    Remark~\ref{rem:wellposed}.) Such an estimate follows immediately using
    integration by parts and the energy
    conservation~\eqref{eq:a-priori-estimate-energy}.    
  \item We refer to~\cite[Sect.~3.1]{BLP08},~\cite[Lem.~3.2]{SS95} or \cite[Ch.~IV,
    \S3]{S08} for the derivation of
    \begin{equation}\label{eq:flux-equality}
      E_y(0,s) + \Flux(y;\tau_0,s) + \int_{\Lambda(0;\tau_0,s)} \partial_t y(t,\x) \cdot  u(t,\x) \dd(t,\x) = E_y(0,\tau_0)
    \end{equation}
    for $0 \leq \tau_0 \leq s < T^\bullet$.  Here,
    \begin{equation*}
      \Flux(y;\tau_0,s) \defn \int_{M_{\tau_0}^s} e(t,\x) \cdot \nu(t,\x) \dd \sigma(t,\x),
    \end{equation*}
    where
    $M_{\tau_0}^s \defn \bigl\{(t,\x) \in \Lambda(0;\tau_0,s) \colon |\x-\x_0| =
    T^\bullet - t\bigr\}$ is the ``mantle'' and $\nu$ the unit outer normal to
    $\Lambda(0;\tau_0,s)$, and the vector field $e$ is given by
    \begin{equation*}
      e(t,\x) \defn \left(\frac{|\partial_t y(t,\x)|^2 + |\nabla y(t,\x)|^2}2 +
        \frac{|y(t,\x)|^6}6,-\partial_t y(t,\x) \nabla y(t,\x)\right).
    \end{equation*}
    Thus, $\Flux(y;\tau_0,s)$ is the energy transferred across $M_{\tau_0}^s$ during
    transition from $D_{\tau_0}$ to $D_s$, and we have $\Flux(y;\tau_0,s) \geq 0$. As
    in the references for the proof, we show that
    $\lim_{\tau_0\nearrow T^\bullet} \Flux(y;\tau_0,s) = 0$.

    Estimating the integral involving $u$ in~\eqref{eq:flux-equality} from below and
    using the energy bound~\eqref{eq:energy-estimate} we derive
    \begin{equation}\label{eq:local-energy-flux-inequality}
      E_y(0,s) + \Flux(y;\tau_0,s) \leq E_y(0,\tau_0) + E_0\|u\|_{\Le^1(\tau_0,s;\Le^2(\Omega))}. 
    \end{equation}
    The nonnegativity of the flux now implies that the function $f$ defined by
    $t \mapsto E_y(0,t) + E_0\|u\|_{\Le^1(t,T^\bullet;\Le^2(\Omega))}$ is
    nonincreasing. Due to the energy bound~\eqref{eq:a-priori-estimate-energy}, it is
    moreover uniformly bounded for $t \in (0,T^\bullet)$, and thus admits a limit
    $\lim_{t\nearrow T^\bullet} f(t)$. Back
    in~\eqref{eq:local-energy-flux-inequality}, we now have
    \begin{equation*}
      0 \leq \Flux(y;\tau_0,s) \leq f(\tau_0) - f(s) \xrightarrow{\tau_0 \nearrow T^\bullet} 0,
    \end{equation*}
    so indeed $\lim_{\tau_0 \nearrow T^\bullet} \Flux(y;\tau_0,s) = 0$.
  \item Lastly, in the proof in~\cite[Prop.~3.3]{BLP08}, a Morawetz identity is used
    which can be formally derived by multiplying the state equation with
    $(t \cdot\partial_ty(t,\x) + \x\cdot\nabla y(t,\x) + y(t,\x))$. The identity is
    then integrated over $\Lambda(0;\tau_0,s)$ and a bound on the
    $\Le^6(P_\x D_{\tau_0})$-norm of $y(\tau_0)$ is derived; this is of course again
    for smooth solutions of the equation and the claim for mild solutions follows by
    approximation. To comply with the line of proof in~\cite{SS95} or~\cite{BLP08},
    we need to make sure that
    \begin{equation*}
      \int_{\Lambda(0;\tau_0,s)} u(t,\x) \cdot \bigl(t \cdot\partial_ty(t,\x) + \x\cdot\nabla y(t,\x) +
      y(t,\x)\bigr) \dd(t,\x)  \xrightarrow{\tau_0 \nearrow 0} 0. 
    \end{equation*}
    (Note that, in order to stay close to the referred works, we have shifted
    $(T^\bullet,\x_0)$ to $(0,0)$ here, so now $\tau_0 \leq s < 0$.) This however
    follows quite immediately from H\"older's inequality and the energy
    bound~\eqref{eq:a-priori-estimate-energy}, as the absolute value of the left-hand
    side can be estimated by
    \begin{multline*}
      \int_{\Lambda(0;\tau_0,s)} \left| u(t,\x) \cdot \bigl(t \cdot\partial_ty(t,\x)
        + \x\cdot\nabla y(t,\x) + y(t,\x)\bigr) \right| \dd(t,\x) \\ \leq
      \int_{\tau_0}^s \|u\|_{\Le^2(\Omega)} \left(|\tau_0|\|\partial_t
        y\|_{\Le^2(\Omega)} +
        \diam(\Omega)\|\nabla y\|_{\Le^2(\Omega)} + \|y\|_{\Le^2(\Omega)}\right) \dd t \\
      \lesssim \|u\|_{\Le^1(\tau_0,0;\Le^2(\Omega))}.
    \end{multline*}
  \end{enumerate}
  With these three modifications, we can now repeat the proof of
  $\Le^6$-non-concentration verbatim as in~\cite[Prop.~3.3]{BLP08} with $u$ inserted
  at the appropriate places.
\end{proof}

To make use of the foregoing Proposition~\ref{prop:l6-nonconcentration}, we next
establish a series of preliminary results. The first one is a technical result which
allows us to localize functions to the ``slices'' $P_\x D_t^\delta$.

\begin{lemma}[{Localization~(\cite[Lem.~3.3]{BLP08})}]
  \label{lem:localization}
  Let $1 \leq p \leq \infty$. For every $\x_0 \in \overline\Omega$ there exist
  numbers $r_{\text{ext}} >0$ and $C_{\text{ext}} \geq 0$ with the following
  significance: For $\delta < r_{\text{ext}}$, there exists
  $t_0 \in (0 \vee {T^\bullet +\delta-r_{\text{ext}}},T^\bullet)$ such that given
  $v \in \Lloc^1(0,T^\bullet;\Le^p(\Omega))$, there exists a function
  $\check v \in \Lloc^1(0,T^\bullet;\Le^p(\Omega))$ such that
  \begin{equation*}
    v(t) = \check v(t) \quad \text{a.e.\ on}~P_\x D_t^\delta
  \end{equation*}
  and
  \begin{equation}
    \label{eq:extension-estimate}
    \|\check v(t)\|_{\Le^p(\Omega)} \leq C_{\text{ext}} \left\|v(t)\right\|_{\Le^p(P_\x D_t^\delta)}
  \end{equation}
  for all $t \in (t_0,T^\bullet)$.
\end{lemma}

The number $r_{\text{ext}}$ in the next proposition is the one from
Lemma~\ref{lem:localization}.

\begin{proposition}
  \label{prop:l-infty-disc-implies-l4l12-cone} Let $0 \leq \delta < r_{\text{ext}}$
  and assume that for every $\eps > 0$ there exists $\tau_0 \in (0,T^\bullet)$ such
  that
  \begin{equation*}
    \|y\|_{\LeLe\infty6(\Lambda(\delta;\tau_0,T^\bullet))} < \eps
  \end{equation*}
  Then there is $t_0 \in (0,T^\bullet)$ such that
  $y \in \LeLe4{12}(\Lambda(\delta;t_0,T^\bullet))$.
\end{proposition}

\begin{proof}
  Let $t_0 \in (0 \vee {T^\bullet +\delta-r_{\text{ext}}},T^\bullet)$ be fixed for
  now, to be chosen later. We use the assumption in conjunction with the Strichartz
  estimate~\eqref{eq:strichartz-estimate}. Let $\check y$ be the function from
  Lemma~\ref{lem:localization} coinciding with $y$ on $\Lambda(\delta;t_0,T^\bullet)$
  and let $w$ be given by
  \begin{equation*}
    \xi_w(t) = e^{\cA t}\xi_y(t_0)
    + \int_{t_0}^t e^{\cA (t-s)} \begin{pmatrix}
      0 \\ u(s) - \check y^5(s)
    \end{pmatrix}
    \dd s,
  \end{equation*}
  so the mild solution to the linear wave equation~\eqref{eq:LWE} on
  $[t_0,T^\bullet)$ with $f = u - \check y$ and initial data $\xi_y(t_0)$. Then $w$
  coincides with $y$ on $\Lambda(\delta;t_0,T^\bullet)$ due to finite speed of
  propagation, cf.\ Remark~\ref{rem:wellposed}. Using the Strichartz
  estimate~\eqref{eq:strichartz-estimate} for this linear equation we obtain
  \begin{align*}
    \|y\|_{\LeLe4{12}(\Lambda(\delta;t_0,T^\bullet))} &\leq
    \|w\|_{\Le^4(t_0,T^\bullet;\Le^{12}(\Omega))} \\ &\lesssim_{T^\bullet} 
    \bigl\|\xi_y(t_0)\bigr\|_{\cE} + \bigl\|\check
    y^5\bigr\|_{\Le^1(t_0,T^\bullet;\Le^2(\Omega))} +  
    \bigl\|u\bigr\|_{\Le^1(t_0,T^\bullet;\Le^2(\Omega))}.
  \end{align*}
  The extension estimate~\eqref{eq:extension-estimate} and the interpolation
  inequality~\eqref{eq:interpolation-l5l10} further yield
  \begin{align*}
    \bigl\|\check y^5\bigr\|_{\Le^1(t_0,T^\bullet;\Le^2(\Omega))} & \lesssim
    \bigl\|y\bigr\|_{\LeLe5{10}(\Lambda(\delta;t_0,T^\bullet))} \\ & \lesssim
    \|y\|_{\LeLe4{12}(\Lambda(\delta;t_0,T^\bullet))}^4  \|y\|_{\LeLe\infty6(\Lambda(\delta;t_0,T^\bullet))}.
  \end{align*}
  Hence, choosing $\eps$ small enough (cf.\ Lemma~\ref{lem:useful-estimate}) and if
  necessary enlarging $t_0$ to $\tau_0(\eps)$, we obtain
  \begin{equation*}
    \|y\|_{\LeLe4{12}(\Lambda(\delta;t_0,T^\bullet))} \lesssim 
    \bigl\|\xi_y(t_0)\bigr\|_{\cE} + 
    \|u\|_{\Le^1(t_0,T^\bullet;\Le^2(\Omega))}.
  \end{equation*}
  An application of the energy bound~\eqref{eq:a-priori-estimate-energy} then yields
  the claim.
\end{proof}

An immediate consequence of Proposition~\ref{prop:l-infty-disc-implies-l4l12-cone}
and its proof together with the interpolation
inequality~\eqref{eq:interpolation-l5l10} is the following:

\begin{corollary}
  \label{cor:l5-l10-small} Let the assumption of
  Proposition~\ref{prop:l-infty-disc-implies-l4l12-cone} hold true for some $\delta$
  satisfying $0 \leq \delta < r_{\text{ext}}$. Then, for every $\eps > 0$, there
  exists $t_0 \in (0,T^\bullet)$ such that
  \begin{equation*}
    \bigl\|y\bigr\|_{\LeLe5{10}(\Lambda(\delta;t_0,T^\bullet))} < \eps. 
  \end{equation*}
\end{corollary}

We will need a bound for the energy transfer from one time level to another in the
light cones when we come close enough to $T^\bullet$. The following lemma states that
this is possible and, crucially, even uniformly in $\delta$.

\begin{lemma}
  \label{lem:flux}
  For every $\eps > 0$ there is $t_0 \in [0,T^\bullet)$ such that
  \begin{equation*}
    E_y(\delta;s)  \leq E_y(\delta;\tau_0) + \eps
  \end{equation*}
  for all $\tau_0,s$ satisfying $t_0 \leq \tau_0 \leq s < T^\bullet$, and all
  $\delta \geq 0$.
\end{lemma}

\begin{proof}
  Let $0 \leq t_0 \leq \tau_0 \leq s < T^\bullet$.  As in the proof of
  Proposition~\ref{prop:l5-l10-small-then-cone-bigger}, we obtain for \emph{every}
  $\delta \geq 0$
  \begin{equation*}
    E_y(\delta,s)  + \int_{\Lambda(\delta;\tau_0,s)} \partial_t y(t,\x) \cdot  u(t,\x) \dd(t,\x) \leq E_y(\delta,\tau_0). 
  \end{equation*}
  and so
  \begin{equation*} E_y(\delta,s) \leq E_y(\delta,\tau_0) +
    E_0\bigl\|u\bigr\|_{\Le^1(t_0,T^\bullet;\Le^2(\Omega))}.
  \end{equation*}
  Choosing $t_0$ sufficiently close to $T^\bullet$, this gives the claim.
\end{proof}

Finally, the next proposition shows that the $\Le^6$-non-concentration effect as
proven in Proposition~\ref{prop:l6-nonconcentration} in fact also holds in
$\alpha$-enlarged light cones for $\alpha > 0$ sufficiently small. This will then
immediately imply the main Theorem~\ref{thm:global-existence} below.

\begin{proposition}
  \label{prop:l5-l10-small-then-cone-bigger} Let the assumption of
  Proposition~\ref{prop:l-infty-disc-implies-l4l12-cone} hold true for $\delta =
  0$. Then, for every $\eps > 0$ there exist $t_0 \in (0,T^\bullet)$ and
  $0<\alpha < r_{\text{ext}}$ such that
  \begin{equation*}
    \|y\|_{\LeLe\infty6(\Lambda(\alpha;t_0,T^\bullet))} < \eps.
  \end{equation*}
\end{proposition}

\begin{proof}
  Let $\eps > 0$. We do explicit estimates to demonstrate that there are no implicit
  dependencies on the choice of $t_0$ along the proof. Via
  Corollary~\ref{cor:l5-l10-small} and Lemma~\ref{lem:flux}, choose
  $\tau_0 \in (0 \vee {T^\bullet -r_{\text{ext}}},T^\bullet)$ such that
  \begin{equation*}
    \bigl\| y\bigr\|_{\LeLe5{10}(\Lambda(0;s,T^\bullet))} < \eps
  \end{equation*}
  and
  \begin{equation*}
    E_y(\delta;s)  \leq E_y(\delta;\eta_0) + C_e(T^\bullet) C_{\text{ext}}^5\eps^5,
  \end{equation*}
  as well as
  \begin{equation*}
    \|u\|_{\Le^1(s,T^\bullet;\Le^2(\Omega))} \leq C_{\text{ext}}^5\eps^5,
  \end{equation*}
  all for all $\tau_0 \leq \eta_0 \leq s < T^\bullet$, and all $\delta \geq 0$, where
  $C_e$ was the constant from Lemma~\ref{lem:estimates}.

  Let again $\check y$ be the function from Lemma~\ref{lem:localization} coinciding
  with $y$ on $\Lambda(0;\tau_0,T^\bullet)$. We split the local solution $y$ on
  $[\tau_0,T^\bullet)$ into a homogeneous part $y_h$ and an inhomogeneous part $y_i$
  by
  \begin{equation*}
    \xi_{y_h}(t) \defn e^{\cA t}\xi_y(\tau_0), \qquad \xi_{y_i}(t) \defn \int_{\tau_0}^t e^{\cA
      (t-s)} 
    \begin{pmatrix}
      0 \\ u(s) - y^5(s)
    \end{pmatrix}
    \dd s.
  \end{equation*}
  With $w_i$ defined by \begin{equation*} \xi_{w_i}(t) \defn \int_{\tau_0}^t e^{\cA
      (t-s)} \begin{pmatrix} 0 \\ u(s) - \check y^5(s)
    \end{pmatrix}
    \dd s
  \end{equation*}
  on $[\tau_0,T^\bullet]$, we have $w_i = y_i$ on
  $\Lambda(0;\tau_0,T^\bullet)$. Thus, the estimates in Lemma~\ref{lem:estimates}
  together with Remark~\ref{rem:norm-solution-operator-monotone}, the choice of
  $\tau_0$, and~\eqref{eq:extension-estimate} imply
  \begin{multline*}
    \|\nabla y_i\|_{\LeLe\infty2(\Lambda(0;\tau_0,T^\bullet))} + \|\partial_t y_i\|_{\LeLe\infty2(\Lambda(0;\tau_0,T^\bullet))} + \|y_i\|_{\LeLe\infty6(\Lambda(0;\tau_0,T^\bullet))} \\
    \leq \|w_i\|_{\C([\tau_0,T^\bullet];\cE)} +
    \|w_i\|_{\Le^\infty(\tau_0,T^\bullet;\Le^6(\Omega))} \\\leq C_e(T^\bullet)
    \left(\bigl\|\check y\bigr\|_{\Le^5(\tau_0,T^\bullet);\Le^{10}(\Omega)}^5 +
      \|u\|_{\Le^1(\tau_0,T^\bullet;\Le^2(\Omega))}\right) < 2C_e(T^\bullet)
    C_{\text{ext}}^5\eps^5.
  \end{multline*}
  For the local energy of the homogeneous part $y_h$ of $y$, we find by
  $y_h \in \C([0,T^\bullet];\cE)$ and conservation of energy
  \begin{multline*}
    \frac12 \int_{P_\x D_t} |\nabla y_h(t,\x)|^2 \dd \x \\ \lesssim \bigl\|\nabla
    y_h(t) - \nabla y_h(T^\bullet)\bigr\|_{\Le^2(\Omega)}^2 + \int_{P_\x D_t} |\nabla
    y_h(T^\bullet,\x)| \dd\x \, \xrightarrow{t \nearrow T^\bullet} \, 0.
  \end{multline*}
  Treating the $\partial_t y_h$ term in $E_{y_h}(0,t)$ analogously, we thus obtain
  \begin{equation*}
    \lim_{t\nearrow T^\bullet}E_{y_h}(0;t) = 0.
  \end{equation*}
  On the other hand, the inhomogeneous part was already estimated by
  \begin{equation*}
    E_{y_i}(0;t) <  2C_e(T^\bullet) C_{\text{ext}}^5\eps^5
  \end{equation*}
  for all $t \in [\tau_0,T^\bullet]$, hence we can choose
  $t_0 \in [\tau_0,T^\bullet)$ to obtain
  \begin{equation*}
    E_{y}(0;t) \leq  2^6\bigl(E_{y_h}(0;t) + E_{y_i}(0;t)\bigr) <  129C_e(T^\bullet) C_{\text{ext}}^5\eps^5
  \end{equation*}
  for all $t \in [t_0,T^\bullet]$. The ``homogeneous energy''
  $\|\xi_y(t_0)\|_\cE + \|y(t_0)\|_{\Le^6(\Omega)}$ is an upper bound for
  $E_{y}(\delta;t_0)$ for every $\delta$, and finite by
  Lemma~\ref{lem:energy-conservation-estimate}. Thus, the dominated convergence
  theorem, used with respect to $\delta$, yields $\alpha = \alpha(t_0) > 0$ such that
  \begin{equation*}
    E_{y}(\alpha;t_0) <  130C_e(T^\bullet) C_{\text{ext}}^5\eps^5
  \end{equation*}
  and $\alpha < r_{\text{ext}}$.  We can then finally make use of the choice of
  $\tau_0$ done at the beginning of the proof and its uniformity w.r.t.\ $\delta$ to
  find
  \begin{multline*}
    \|y\|_{\LeLe\infty6(\Lambda(\alpha;t_0,T^\bullet))} \leq \sup_{t_0 \leq t <
      T^\bullet} E_{y}(\alpha;t) \\\leq E_{y}(\alpha;t_0) + C_e(T^\bullet)
    C_{\text{ext}}^5\eps^5 < 131C_e(T^\bullet) C_{\text{ext}}^5\eps^5.
  \end{multline*}
  This completes the proof.
\end{proof}

\begin{theorem}[Global existence]
  \label{thm:global-existence}
  For every $u \in \Le^1(0,\sT;\Le^2(\Omega))$, the local solution $y$
  to~\eqref{eq:main-equation} as given in Theorem~\ref{thm:local-existence} exists
  globally in time on the interval $[0,\sT]$ and satisfies
  $\xi_y \in \C([0,\sT];\cE)$ and $y \in \Le^4(0,\sT;\Le^{12}(\Omega))$.
\end{theorem}

\begin{proof}
  We had already noted that it is sufficient to show that the local solution $y$
  satisfies $y \in \Le^4(0,T^\bullet;\Le^{12}(\Omega))$ since this allows to show
  that $\lim_{t\nearrow T^\bullet} \xi_y(t)$ exists in $\cE$ via the variation of
  constants formula and the estimates as in Lemma~\ref{lem:estimates} together with
  the energy conservation~\eqref{eq:a-priori-estimate-energy}. Now, for this purpose,
  let $\x_0 \in \overline\Omega$ be fixed. Proposition~\ref{prop:l6-nonconcentration}
  tells us that the premise of Proposition~\ref{prop:l-infty-disc-implies-l4l12-cone}
  is satisfied for $\delta = 0$. From there,
  Proposition~\ref{prop:l5-l10-small-then-cone-bigger} implies, again via
  Proposition~\ref{prop:l-infty-disc-implies-l4l12-cone}, that there are
  $t_0 \in (0,T^\bullet)$ and $\alpha > 0$ such that
  $y \in \LeLe4{12}(\Lambda(\alpha;t_0,T^\bullet))$.

  This can be done for every $\x_0 \in \overline\Omega$. Then, the collection of sets
  $\bigl(P_\x D_{T^\bullet}^{\alpha(\x_0)}\bigr)_{\x_0 \in\overline\Omega}$ is a
  (relatively) open covering of $\overline \Omega$. Compactness of the latter gives a
  finite set of points $\x_i \in \overline\Omega$, $i=1,\dots,n$, such that
  $\bigl(P_\x D_{T^\bullet}^{\alpha(\x_i)}\bigr)_{i=1,\dots,n}$ is still a
  (relatively) open covering of $\overline\Omega$. Setting
  $t^\star_0 \defn \max_{i=1,\dots,n} t_0(\x_i)$, we find
  $y \in \Le^4(t^\star_0,T^\bullet;\Le^{12}(\Omega))$, and since we already knew that
  $y \in \Lloc^4([0,T^\bullet);\Le^{12}(\Omega))$, this gives
  $y \in \Le^4(0,T^\bullet;\Le^{12}(\Omega))$ as desired.
\end{proof}

\section{Optimal control} \label{sec:optimal-control}

We recall the setup of the optimal control problem. Let
$y_d \in \Le^2(\Omega)$ and nonnegative scaling parameters
$\gamma,\beta_1,\beta_2$ be given. We had
\begin{multline*}
  \ell(y,u) 
  \defn \frac12\|y(\sT) - y_d\|_{\Le^2(\Omega)}^2 +
  \frac\gamma4\|y\|_{\Le^4(0,\sT;\Le^{12}(\Omega))}^4 \\
  + \beta_1 \|u\|_{\Le^1(0,\sT;\Le^2(\Omega))} + \frac{\beta_2}2
  \|u\|_{\Le^2(0,\sT;\Le^2(\Omega))}^2
\end{multline*}
for $y \in \C([0,\sT];\Le^2(\Omega)) \cap \Le^4(0,\sT;\Le^{12}(\Omega))$ and
$u \in \Le^r(0,\sT;\Le^2(\Omega))$, where $r = 1$ if $\beta_2 = 0$ and $r = 2$ if
$\beta_2 > 0$. We consider $\ell$ as a cost functional or performance index
for~\eqref{eq:main-equation}, resulting in the associated optimal control problem
\begin{equation*}
  \optiprog{\min_{y,u}}{\ell(y,u)}{u \in \uad, \\
    y~\text{is the solution to~\eqref{eq:main-equation}.}}{\ref{eq:OCP}} 
\end{equation*}
Here, $\uad$ is a closed and convex, and thus weakly closed, nonempty set of the form
\begin{equation*}
  \uad \defn \Bigl\{v \in \Le^r(0,\sT;\Le^2(\Omega)) \colon
  \|v(t)\|_{\Le^2(\Omega)} \leq  \omega(t)~\text{f.a.a.}~t
  \in (0,\sT)\Bigr\}
\end{equation*}
for a measurable function $\omega$ which is nonnegative almost everywhere on
$(0,\sT)$. We emphasize once more that $\omega$ is not assumed to be bounded away from $0$
uniformly almost everywhere.
Further, of course, the solution
$y$ in~\eqref{eq:OCP} is meant in the sense of Theorem~\ref{thm:global-existence}.

We will proceed to establish existence of globally optimal solutions
to~\eqref{eq:OCP} in the following. Moreover, we will give necessary optimality
conditions of both first and second order, and also second order sufficient conditions.

\subsection{Existence of globally optimal controls}

It now becomes convenient that a solution $y$ associated to $u$---which we denote by
$y_u$ from now on---in the sense of Theorem~\ref{thm:global-existence} is a
Shatah-Struwe solution as noted in Lemma~\ref{lem:weak-and-mild}:

\begin{theorem}[Existence of optimal controls]
  \label{thm:existence-OC} Let $\beta_2 > 0$ or let
  $\omega \in \Le^1(0,\sT)$, and let $\gamma > 0$. Then the optimal control
  problem~\eqref{eq:OCP} admits at least one globally optimal pair
  $(y_{\bar u},\bar u)$ with $\bar u \in \uad$ such that the state $y_{\bar u}$ is
  the unique global Shatah-Struwe solution to~\eqref{eq:main-equation} for the
  right-hand side $\bar u$.
\end{theorem}

\begin{proof}
  Since $\uad \neq \emptyset$, and $\ell$ is bounded from below by zero, we obtain an
  infimal sequence $(y_k,u_k)$ with $(u_k) \subseteq \uad$, such that $\ell(y_k,u_k)$
  tends to $\inf_{u \in \uad} \ell(y_u,u) > - \infty$ as $k$ goes to infinity, where
  $y_k \defn y_{u_k}$. Due to the assumptions, the sequence $(u_k)$ admits a
  subsequence, denoted by the same name, which converges weakly in
  $\Le^r(0,\sT;\Le^2(\Omega))$ to some limit $\bar u$. (We will keep the notation for
  all convergent subsequences in the following.) Indeed, if $\beta_2 > 0$, this is
  true because then the sequence $(u_k)$ is bounded in
  $\Le^2(0,\sT;\Le^2(\Omega))$. If
  $\omega \in \Le^1(0,\sT)$, then $\uad$ is in fact weakly compact in
  $\Le^1(0,\sT;\Le^2(\Omega))$, cf.~\cite[Cor.~2.6]{DRS93}. Due to weak closedness of
  $\uad$, we also have $\bar u \in \uad$.

  We turn to $(y_k)$: The boundedness of $(u_k)$ implies that $(\xi_{y_k})$ must be
  bounded in $\Le^{\infty}(0,\sT;\cE)$ by the energy
  bound~\eqref{eq:a-priori-estimate-energy}. This gives a weakly-$\ast$ convergent
  subsequence of $(\xi_{y_k})$ with the weak-$\ast$ limit denoted by
  $\bar y \in \Le^{\infty}(0,\sT;\cE)$. We need to show that $\bar y = y_{\bar u}$.

  Looking at the definition of a Shatah-Struwe solution
  \begin{multline}
    -\int_0^\sT \bip{\partial_t y_k(t),\partial_t \varphi(t)} \dd t + \int_0^\sT
    \bip{\nabla y_k(t),\nabla\varphi(t)} \dd t + \int_0^\sT
    \bip{y_k^5(t),\varphi(t)}\dd t \\ = \int_0^\sT \bip{u_k(t),\varphi(t)} \dd t
    \qquad \text{for all}~\varphi \in \C_c^\infty\bigl((0,\sT) \times
    \Omega\bigr),\label{eq:shatah-struwe-definition}
  \end{multline}
  we observe that the linear terms are already dealt with. It remains to show that in
  fact $\bar y \in \Le^4(0,\sT;\Le^{12}(\Omega))$, that $\xi_{\bar y}(0) = \xi_0$,
  and that the nonlinear term involving $y_k^5$ converges to the correct one
  involving $\bar y^5$.

  For the latter, let $5 \leq p < 6$. Then boundedness of $(\xi_{y_k})$ in
  $\Le^{\infty}(0,\sT;\cE)$ and compactness of the embedding
  $\H^1_0(\Omega) \embedsc \Le^p(\Omega)$ implies that $(y_k)$ is in fact a
  pre\emph{compact} set in $\C([0,\sT];\Le^p(\Omega))$,
  see~\cite[Cor.~4]{S87}. Accordingly, there is yet another subsequence of $(y_k)$
  denoted by the same name such that $(y_k)$ tends to $\bar y$ in that space. This
  further means that $(y_k^5)$ converges to $\bar y^5$ in
  $\C([0,\sT];\Le^{p/5}(\Omega))$.

  The next step is to show that $\xi_{\bar y}(0) = \xi_0$.  We already know that
  $(\xi_{y_k})$ is bounded in
  $\Le^\infty(0,\sT;\cE)$. Further,~\eqref{eq:weak-solution-second-deriv} and
  boundedness of $(u_k)$ in $\Le^1(0,\sT;\Le^2(\Omega))$ show that
  $(\partial_t^2 y_k)$ is bounded in $\Le^1(0,\sT;\H^{-1}(\Omega))$. Using
  again~\cite[Cor.~4]{S87}, we infer that $(\xi_{y_k})$ is precompact in
  $\C([0,\sT];\Le^2(\Omega) \times \H^{-1}(\Omega))$ and thus admits a convergent
  subsequence in that space, with the same name and the limit $\bar y$. Thus, in
  particular,
  \begin{equation*}
    \xi_0 = \xi_{y_k}(0) \longrightarrow \xi_{\bar y}(0) \quad
    \text{in}~\H^{-1}(\Omega) \times \Le^2(\Omega) 
  \end{equation*}
  and we obtain $\xi_{\bar y}(0) = \xi_0$, which by assumption even lies in $\cE$.

  We are now prepared to taking limits in~\eqref{eq:shatah-struwe-definition} and
  obtain that 
  $\xi_{\bar y} \in \Le^\infty(0,\sT;\cE)$ satisfies $\xi_{\bar y}(0) = \xi_0$ and
  \begin{multline}\label{eq:weak-solution-satisfy}
    -\int_0^\sT \bip{\partial_t \bar y(t),\partial_t \varphi(t)} \dd t + \int_0^\sT
    \bip{\nabla \bar y(t),\nabla\varphi(t)} \dd t + \int_0^\sT \bip{\bar
      y^5(t),\varphi(t)}\dd t \\ = \int_0^\sT \bip{\bar u(t),\varphi(t)} \dd t \qquad
    \text{for all}~\varphi \in \C_c^\infty\bigl((0,\sT) \times \Omega\bigr).
  \end{multline}

  It only remains to show that $\bar y \in \Le^4(0,\sT;\Le^{12}(\Omega))$. From
  $\gamma > 0$ we infer that $(y_k)$ is also bounded in
  $\Le^4(0,\sT;\Le^{12}(\Omega))$. Hence, there exists a weakly convergent
  subsequence with the limit $\hat y$ in
  $\Le^4(0,\sT;\Le^{12}(\Omega))$. This means that we have
  \begin{multline*}
    \int_0^\sT \bip{\bar y(t),\phi(t)} \dd t = \int_0^\sT \bip{\hat y(t),\phi(t)} \dd t \\
    \text{for all}~\phi \in \Le^1(0,\sT;\Le^{6/5}(\Omega)) \cap
    \Le^{4/3}(0,\sT;\Le^{12/11}(\Omega))
  \end{multline*}
  and thus
  \begin{equation*}
    \bar y = \hat y \quad \text{in} \quad \Le^\infty(0,\sT;\Le^6(\Omega)) +
    \Le^4(0,\sT;\Le^{12}(\Omega)) \embeds \Le^4(0,\sT;\Le^6(\Omega)).
  \end{equation*}
  But then $\bar y = \hat y$ almost everywhere in $(0,\sT) \times \Omega$ and in fact
  $\bar y \in \Le^4(0,\sT;\Le^{12}(\Omega))$.

  Since $\bar y$ has now been shown to be a Shatah-Struwe solution
  to~\eqref{eq:main-equation}, uniqueness of such solutions as established in
  Corollary~\ref{cor:mild-l4l12-unique} then finally implies that indeed
  $\bar y = y_{\bar u}$.

  It is now standard to use the convergences $(y_k) \to \bar y$ in
  $\C([0,\sT];\Le^2(\Omega))$ and $(y_k) \wto \bar y$ in
  $\Le^4(0,\sT;\Le^{12}(\Omega))$ as well as $(u_k) \wto \bar u$ in
  $\Le^1(0,\sT;\Le^2(\Omega))$ together with lower semicontinuity of norms in the
  objective $\ell$ to infer that
  \begin{equation*}
    \ell(\bar y,\bar u) = \inf_{u\in\uad}\ell(y_u,u).
  \end{equation*}
  This shows that there indeed exists a globally optimal control to~\eqref{eq:OCP}.
\end{proof}


From the proof of Theorem~\ref{thm:existence-OC} we obtain the following
auxiliary result for the case $\gamma = 0$, so the case where there is no additional
$\Le^4(0,\sT;\Le^{12}(\Omega))$ norm term in the objective. It underlines the role of
this term in upgrading weak solutions to (unique) mild solutions:

  \begin{proposition}
    Let $\beta_2 > 0$ or $\omega \in \Le^1(0,\sT)$. Then the optimal control
    problem~\eqref{eq:OCP} admits at least one globally optimal pair
    $(\bar y,\bar u)$ with $\bar u \in \uad$ such that $\bar y$ is a---possibly
    non-unique---weak solution to~\eqref{eq:main-equation} with right-hand side
    $\bar u$. That means we have $\xi_{\bar y}(0) = \xi_0$ and
    $\xi_{\bar y} \in \Le^\infty([0,\sT];\cE)$, and $\bar y$ satisfies the weak
    formulation~\eqref{eq:weak-solution-satisfy}.
  \end{proposition}

  \subsection{Optimality conditions}

  Let us set
  \begin{equation*}
    \cY \defn \Bigl\{y \colon \xi_y \in \C([0,\sT];\cE)\Bigr\}
    \quad \text{and} \quad \cY_+ \defn  \Le^4(0,\sT;\Le^{12}(\Omega)) \cap \cY
  \end{equation*}
  with
  \begin{equation*}
    \|y\|_\cY \defn \|\xi_y\|_{\C([0,\sT];\cE)} \quad \text{and}
    \quad \|y\|_{\cY_+} \defn \|y\|_{\cY} + \|y\|_{ \Le^4(0,\sT;\Le^{12}(\Omega))}.
  \end{equation*}

  As a first step towards optimality conditions, we show that the control-to-state
  mapping $u\mapsto y_u$ is twice continuously differentiable from
  $\Le^1(0,\sT;\Le^2(\Omega))$ to $\cY_+$. We recall the definition of $y_u$ in terms
  of the variation-of-constants (or Duhamel) formula
  in~\eqref{eq:variation-of-constants} and define the mapping
  \begin{equation*}
    e \colon \cY_+ \times \Le^1(0,\sT;\Le^2(\Omega)) \to \C([0,\sT];\cE)
  \end{equation*}
  as follows:
  \begin{equation*}
    \bigl[e(y,u)\bigr](t) \defn e^{\cA t}\xi_0 + \int_0^t e^{\cA (t-s)} \begin{pmatrix}
      0 \\ u(s) - y^5(s)
    \end{pmatrix}
    \dd s - \xi_y(t).
  \end{equation*}

  By construction it is clear that $e(y_u,u) = 0$ and of course it is our first goal
  to use the implicit function theorem to show the following:

  \begin{theorem}
    \label{thm:control-to-state-diff}
    The control-to-state operator $\cS \colon u \mapsto y_u$ is twice continuously
    differentiable from $\Le^1(0,\sT;\Le^2(\Omega))$ to $\cY_+$. Its derivative
    $\cS'(\bar u)h$ in $\bar u$ in direction $h \in \Le^1(0,\sT;\Le^2(\Omega))$ is
    given by the mild solution $z_h \in \cY_+$ of
    \begin{equation*}
      \begin{aligned}
        \partial_t^2 z - \Delta z + 5\bar y^4_{\bar u}z  & = h && \text{in}~(0,T) \times \Omega, \\
        z & = 0 && \text{on}~(0,T) \times \partial\Omega, \\
        \bigl(z(0),\partial_t z(0)\bigr) & = (0,0) && \text{in}~\Omega,
      \end{aligned}
    \end{equation*}
    on $[0,\sT]$, i.e., $z_h$ satisfies
    \begin{equation*}
      \xi_z(t) = \int_0^t e^{\cA (t-s)} \begin{pmatrix}
        0 \\ h(s) - 5y_{\bar u}^4(s)z(s)
      \end{pmatrix}
      \dd s
    \end{equation*}
    for all $t \in [0,\sT]$. Its second derivative $\cS''(\bar u)(h_1,h_2)$ in
    $\bar u$ in directions $h_1,h_2 \in \Le^1(0,\sT;\Le^2(\Omega))$ is given by the
    mild solution $z_{h_1,h_2} \in \cY_+$ of
    \begin{equation*}
      \begin{aligned}
        \partial_t^2 z - \Delta z + 5\bar y_{\bar u}^4 z & = -20\bar y^3_{\bar u}z_{h_1}z_{h_2} && \text{in}~(0,T) \times \Omega, \\
        z & = 0 && \text{on}~(0,T) \times \partial\Omega, \\
        \bigl(z(0),\partial_t z(0)\bigr) & = (0,0) && \text{in}~\Omega,
      \end{aligned}
    \end{equation*}    
    on $[0,\sT]$, where $z_{h_i} = \cS'(\bar u)h_i$ for $i=1,2$, i.e., $z_{h_1,h_2}$
    satisfies
    \begin{equation*}
      \xi_z(t) = \int_0^t e^{\cA (t-s)} \begin{pmatrix}
        0 \\  -20\bar y^3_{\bar u}(s)z_{h_1}(s)z_{h_2}(s)- 5y_{\bar u}^4(s)z(s)
      \end{pmatrix}
      \dd s
    \end{equation*}
    for all $t \in [0,\sT]$.
  \end{theorem}

\begin{proof}
  We begin by showing that $e$ is twice continuously differentiable. Clearly,
  $y \mapsto \xi_y$ is a continuous linear mapping from $\cY$ into $\C([0,\sT];\cE)$,
  just as
  \begin{equation}
    v \mapsto \left[t\mapsto\int_0^t e^{\cA (t-s)} \begin{pmatrix}
        0 \\ v(s)
      \end{pmatrix}
      \dd s\right]\label{eq:inhomogeneous-solution-cont-dependence}
  \end{equation}
  is a continuous linear mapping from $\Le^1(0,\sT;\Le^2(\Omega))$ to
  $\C([0,\sT];\cE)$. It is thus sufficient to show that $y \mapsto y^5$ is twice
  continuously differentiable considered as a mapping from $\cY_+$ to
  $\Le^1(0,\sT;\Le^2(\Omega))$. This, however, follows immediately from the
  interpolation inequality~\eqref{eq:interpolation-l5l10}, which implies that
  \begin{equation}\label{eq:solution-space-embed-l5l10}
    \cY_+ \embeds \Le^4(0,\sT;\Le^{12}(\Omega)) \cap \Le^\infty(0,\sT;\Le^6(\Omega)) \embeds \Le^5(0,\sT;\Le^{10}(\Omega)),
  \end{equation}
  together with twice continuous differentiability of the Nemytskii operator induced
  by the real function $x \mapsto x^5$ between $\Le^5(0,\sT;\Le^{10}(\Omega))$ and
  $\Le^1(0,\sT;\Le^2(\Omega))$, cf.~\cite[Thms.~7\&9]{GKT92}. Altogether $e$ is twice
  continuously differentiable.
  
  In order to use the implicit function theorem, it remains to show that $e_y(y_u,u)$
  is continuously invertible as a linear operator between $\cY_+$ and
  $\C([0,\sT];\cE)$ for every $u \in \Le^1(0,\sT;\Le^2(\Omega))$. From the foregoing
  considerations and the chain rule, we obtain
  \begin{equation*}
    \bigl[e_y(y_u,u)z\bigr](t) =
    -\int_0^t e^{\cA(t-s)}
    \begin{pmatrix}
      0 \\ 5 y_u^4(s) z(s)
    \end{pmatrix} \dd s - \xi_z(t).
  \end{equation*}
  Thanks to the open mapping theorem, it will be sufficient to prove that for every
  $F \in \C([0,\sT];\cE)$ there is a unique $z \in \cY_+$ such that
  $e_y(y_u,u)z = F$. We can again use a fixed point theorem to show that this is the
  case: Choose a partition $0 = t_0 < t_1 <\dots < t_n = \sT$ such that
  \begin{equation*}
    5  \bigl(C_e(\sT) + C_s(\sT)\bigr) C_{\text{em}} \|y_u\|_{\Le^4(t_i,t_{i+1};\Le^{12}(\Omega))}^4< \frac12 \quad
    \text{for all}~i=0,\dots,n-1,
  \end{equation*}
  where $C_{\text{em}}$ is the embedding constant of
  $\H^1_0(\Omega) \embeds \Le^6(\Omega)$; for the constants $C_s,C_e$, see
  Lemma~\ref{lem:estimates}. Let $\cY(t_i,t_{i+1})$ and $\cY_+(t_i,t_{i+1})$ be the
  spaces $\cY$ and $\cY_+$ on the interval $[t_i,t_{i+1}]$, \emph{mutatis
    mutandis}. Let further $\xi^i \in \cE$ for $i=0,\dots,n-1$ be given and consider
  the mappings $\cT_i \colon \cY(t_i,t_{i+1}) \to \cY(t_i,t_{i+1})$ defined by
  $\cT_i h = z$ such that
  \begin{equation*}
    \xi_z(t) =  e^{\cA (t-t_i)}\xi^i - \int_{t_i}^t e^{\cA(t-s)}
    \begin{pmatrix}
      0 \\ 5 y_u^4(s) h(s)
    \end{pmatrix} \dd s - F(t) \quad \text{for}~t \in [t_i,t_{i+1}].
  \end{equation*}
  Then, by the estimates~\eqref{eq:energy-estimate}
  and~\eqref{eq:strichartz-estimate}, semigroup properties and H\"older's inequality,
  \begin{multline*}
    \bigl\|\cT_i h_1 - \cT_i h_2\bigr\|_{\cY_+(t_i,t_{i+1})} \\ \leq 5\bigl(C_e(\sT)
    + C_s(\sT)\bigr) C_{\text{em}} \|y_u\|_{\Le^4(t_i,t_{i+1};\Le^{12}(\Omega))}^4
    \bigl\|h_1-h_2\bigr\|_{\C([t_i,t_{i+1}];\H^1_0(\Omega))}.
  \end{multline*}
  The choice of the partition $(t_i)$ and Banach's fixed point theorem tell us that
  every mapping $\cT_i$ possesses a unique fixed point $z_i \in \cY_+(t_i,t_{i+1})$,
  still depending on $\xi^i$. If we iteratively choose $\xi^0 = 0$ and
  $\xi^i = z_{i-1}(t_i)$ for $i = 1,\dots,n-1$ and glue together the resulting
  functions $z_i$ to a function $z \in \cY_+$, then we obtain
  \begin{equation}
    \xi_z(t) +  \int_0^t e^{\cA(t-s)}
    \begin{pmatrix}
      0 \\ 5 y_u^4(s) z(s)
    \end{pmatrix} \dd s = - F(t) \quad \text{for
      all}~t\in[0,\sT],\label{eq:formula-ey-derivative}
  \end{equation}
  i.e., $z \in \cY_+$ satisfies $e_y(y_u,u)z = F$. Thus,
  $e_y(y_u,u)^{-1} \in \cL(\C([0,\sT];\cE);\cY_+)$.

  Finally, the expression for the derivative $\cS'(\bar u)$ comes from the well known
  formula
  \begin{equation*}
    \cS'(\bar u)h = -e_y(\cS(\bar u),\bar u)^{-1} e_u(\cS(\bar u),\bar u)h,
  \end{equation*}
  the observation that $e_u(\cS(\bar u),\bar u)$ is given exactly
  by~\eqref{eq:inhomogeneous-solution-cont-dependence}, and plugging this
  into~\eqref{eq:formula-ey-derivative} for $F$. For the second derivative
  $\cS''(\bar u)$, we take another derivative in the foregoing expression and use
  \begin{equation*}
    \bigl[e_{yy}(y_u,u)(z_1,z_2)\bigr](t) = -\int_0^t e^{\cA(t-s)}
    \begin{pmatrix}
      0 \\ 20 y_u^3(s) z_1(s) z_2(s)
    \end{pmatrix} \dd s.
  \end{equation*}
  This gives the claim.
\end{proof}

As usual, the control-to-state operator $\cS$ allows us to define the \emph{reduced
  problem} which we consider from now on:

\begin{equation}
  \min_{u \in \uad}\ell_r(u), \tag{ROCP} \label{eq:ROCP}
\end{equation}
where we set $\ell_r(u) \defn \ell(y_u,u)$. We decompose the objective function
$\ell_r$ further into
\begin{equation}\label{eq:red-obj-decompose}
  \ell_r(u) \defn F(u) + \beta_1 \sj(u)
\end{equation}
with
\begin{equation*}
  F(u) = \frac12\|y_u(\sT) - y_d\|_{\Le^2(\Omega)}^2 +
  \frac\gamma4\|y_u\|_{\Le^4(0,\sT;\Le^{12}(\Omega))}^4  
  + \frac{\beta_2}2 \|u\|_{\Le^2(0,\sT;\Le^2(\Omega))}^2, 
\end{equation*}
which is smooth as we see below, and the
non-differentiable part
\begin{equation*}
  \sj(u) \defn  \|u\|_{\Le^1(0,\sT;\Le^2(\Omega))}.
\end{equation*}
For the following derivation of necessary and sufficient optimality conditions, we use several
ideas and results from~\cite{CHW17}. The main difference between this work and the
present one is that the constraints on $u$ in~\cite{CHW17} are classical box constraints.

We first quote the following result; it is a
characterization of the subdifferential $\partial\sj$ and a formula for the directional derivative
of $\sj$:

\begin{proposition}[{\cite[Prop.~3.8]{CHW17}}]\label{prop:L1L2-subdiff-character}
  Let $u \in \uad$ and $\lambda \in \Le^\infty(0,\sT;\Le^2(\Omega))$. Then the
  following equivalence holds true:
  \begin{equation*}
    \lambda \in\partial\sj(u) \quad \iff \quad\text{f.a.a.}~t\in[0,\sT]\colon\begin{cases}
      \lambda(t) = \frac{u(t)}{\|u(t)\|_{\Le^2(\Omega)}} & \text{if}~\|u(t)\|_{\Le^2(\Omega)} \neq 0, \\[0.5em]
      \|\lambda(t)\|_{\Le^2(\Omega)} \leq 1 &  \text{if}~\|u(t)\|_{\Le^2(\Omega)} = 0.
    \end{cases}
  \end{equation*}
  Moreover, the directional derivative of $\sj$ in $u \in \Le^1(0,\sT;\Le^2(\Omega))$
  exists in every direction $v \in \Le^1(0,\sT;\Le^2(\Omega))$ and is given by
  \begin{equation*}
    \sj'(u;v) = \int_{\bigl[\|u\|_{\Le^2(\Omega)} = 0\bigr]} \|v(t)\|_{\Le^2(\Omega)} \dd t +
    \int_{\bigl[\|u\|_{\Le^2(\Omega)} \neq 0\bigr]} 
    \frac{\bip{v(t),u(t)}}{\|u(t)\|_{\Le^2(\Omega)}} \dd t.
  \end{equation*}
\end{proposition}

We next establish that $F$ is twice continuously differentiable. For a concise form
of its derivatives, it will be useful to
define the adjoint state:

\begin{definition}[Adjoint state]
  \label{def:adjoint-state}
  Given $\bar u \in \uad$, we denote by $\bar p$ the \emph{adjoint state} defined by
  \begin{equation*}
    \bar p \defn \cS'(\bar u)^\ast\bigl(\delta_\sT^\ast(y_{\bar u}(\sT) - y_d) + \gamma\psi_{\bar
      u}\bigr) \in \Le^\infty(0,\sT;\Le^2(\Omega)).
  \end{equation*}
  Here $\psi_{\bar u} \in \Le^{4/3}(0,\sT;\Le^{12/11}(\Omega))$ is given by
  \begin{equation}\label{eq:l4l12-derivative}
    \psi_{\bar u}(t) \defn \|y_{\bar u}(t)\|_{\Le^{12}(\Omega)}^{-8}|y_{\bar u}(t)|^{10}y_{\bar
      u}(t)
  \end{equation}
  and $\delta_\sT^\ast$ is the adjoint operator of the (continuous linear) point
  evaluation $\delta_\sT$ from $\C([0,\sT];\Le^2(\Omega))$ to $\Le^2(\Omega)$.
\end{definition}

The function $\psi_{\bar u}$ is the $\Le^2$-gradient of
$\Psi(y) \defn \frac14\|y\|_{\Le^4(0,\sT;\Le^{12}(\Omega)}^4$ in
$y_{\bar u}$, that is,
\begin{equation}\label{eq:l4l12-power-deriv}
  \Psi'(y_{\bar u})h = \int_0^\sT \bip{\psi_{\bar u}(t),h(t)}
  \dd t \quad \text{for all}~h\in \Le^4\bigl(0,\sT;\Le^{12}(\Omega)\bigr).
\end{equation}
 This is shown in Corollary~\ref{cor:psi-diff} in the
 appendix. This corollary is also important in the next and final
 result for this preparationary subsection:
 
\begin{lemma}
  \label{lem:f-c2}
  The first summand $F$ of the reduced objective function $\ell_r$ as
  in~\eqref{eq:red-obj-decompose} is twice continuously differentiable
  from $\Le^r(0,\sT;\Le^2(\Omega))$ to $\R$. Its derivatives in
  $\bar u$ are given by
  \begin{equation*}
    F'(\bar u)v = \int_0^\sT \bip{\bar p(t) + \beta_2\bar u(t),v(t)} \dd t
  \end{equation*}
  and
  \begin{multline*}
    F''(\bar u)v^2 = \bigl\|z_v(\sT)\bigr\|_{\Le^2(\Omega)}^2 - \int_0^\sT \bip{\bar p(t),20
      y_{\bar u}^3(t) z_v^2(t)} \dd t \\ + \gamma\Psi''(y_{\bar u})z_v^2 + \beta_2 \bigl\|v\bigr\|_{\Le^2(0,\sT;\Le^2(\Omega))}^2
  \end{multline*}  
  for all $v \in \Le^r(0,\sT;\Le^2(\Omega))$, where $\bar p$ is the adjoint state and
  $z_v = \cS'(\bar u)v$.

  Moreover, the quadratic form $v \mapsto F''(\bar u)v^2$ is
  weakly lower semicontinuous.
\end{lemma}

We recall that $r = 1$ if $\beta_2 = 0$ and $r=2$ otherwise.

\begin{proof}[Proof of Lemma~\ref{lem:f-c2}]
  Corollary~\ref{cor:psi-diff} in the appendix shows that $\Psi$
  is twice continuously differentiable on the whole
  $\Le^4(0,\sT;\Le^{12}(\Omega))$, with the first derivative as
  in~\eqref{eq:l4l12-power-deriv}. The remaining
  differentiability assertions for $F$ and the formula for
  $F''(\bar u)$ are derived by routine calculations.

  The remainder of this proof is devoted to verifying the weak lower
  semicontinuity of the quadratic form induced by $F''(\bar u)$. Let
  $v_k \wto v$ in $\Le^r(0,\sT;\Le^2(\Omega))$. It is clear that the
  quadratic $\Le^2(0,\sT;\Le^2(\Omega))$ norm term in $F''(\bar u)v^2$
  is weakly lower semicontinuous in $v$. We show that the remaining
  terms are even weakly continuous as functions in $v$. From
  $v_k \wto v$ in $\Le^r(0,\sT;\Le^2(\Omega))$ it follows that
  $z_k \defn \cS'(\bar u)v_k \wto z_v \nfed \cS'(\bar u)v$ in $\cY_+$
  by Theorem~\ref{thm:control-to-state-diff}.
  This implies that, cf.~\eqref{eq:solution-space-embed-l5l10},
  \begin{equation}\label{eq:z_k-uniform}
    \|z_k\|_{\Le^5(0,\sT;\Le^{10}(\Omega))} \lesssim \|z_k\|_{\cY} 
    \lesssim 1.
  \end{equation}
  By the compact embedding $\cY_+ \embeds \C([0,\sT];\Le^2(\Omega))$
  as derived from~\cite[Cor.~4]{S87}), we further have
  $z_k \to z_v$ in $\C([0,\sT];\Le^2(\Omega))$. Hence,
  $z_k(t) \to z_v(t)$ in $\Le^2(\Omega)$ for every $t
  \in[0,\sT]$, and in particular
  $\|z_k(\sT)\|_{\Le^2(\Omega)}^2 \to
  \|z_v(\sT)\|_{\Le^2(\Omega}^2$ as $k \to \infty$. Moreover, there is a
  subsequence $(z_{k_\ell})$ of $(z_k)$ such that $z_{k_\ell}^2$
  converges to $z_v^2$ pointwise almost everywhere on
  $(0,\sT) \times \Omega$. H\"older's inequality yields
  \begin{multline*}
    \int_{E_t} \int_{E_\x} \left|\bar p(t,\x)\, y_{\bar
        u}^3(t,\x) \, z_{k_\ell}^2(t,\x) \right| \dd \x \dd t \\
    \leq \bigl\|p\bigr\|_{\Le^\infty(0,\sT;\Le^2(\Omega))}
    \bigl\|y_{\bar u}\bigr\|_{\Le^5(E_t;\Le^{10}(E_\x))}^3
    \bigl\|z_{k_\ell}\bigr\|_{\Le^5(0,\sT;\Le^{10}(\Omega))}^2
  \end{multline*}
  for every measurable subset $E_t \times E_\x$ of
  $(0,\sT) \times \Omega$. Due to~\eqref{eq:z_k-uniform}, the
  integral on the left-hand side thus goes to zero uniformly in
  $\ell$ as $|E_t \times E_\x| \to 0$. We infer that the
  functions $(\bar p y_{\bar u}^3 z_{k_\ell}^2)$ are
  uniformly integrable and the Vitali convergence
  theorem~(\cite[Thm.~III.3.6.6]{DS58}) implies that
  \begin{equation*}
    - \int_0^\sT \bip{\bar p(t),20
      y_{\bar u}^3(t) z_{k_\ell}^2(t)} \dd t \quad \xrightarrow{~\ell\to\infty~} \quad
    - \int_0^\sT \bip{\bar p(t),20 
      y_{\bar u}^3(t) z_v^2(t)} \dd t.
  \end{equation*}
  A subsequence-subsequence argument shows that this convergence
  in fact holds true for the whole sequence $(z_k)$.

  We next turn to the sequence $(\Psi''(y_{\bar u})z_k^2)$ which is
  the last term in $F''(\bar u)v_k^2$. According to
  Corollary~\ref{cor:psi-diff} with $p=4$ and $q=12$, we have
  \begin{multline}\label{eq:psi-second-deriv-formula}
    \Psi''(y_{\bar u})z_k^2 = 11 \int_0^\sT \|y_{\bar
      u}(t)\|^{-8}_{\Le^{12}(\Omega)} \bip{|y_{\bar
        u}(t)|^{10},z_k^2(t)}\dd t \\ - 8 \int_0^\sT \|y_{\bar
      u}(t)\|_{\Le^{12}(\Omega)}^{-20} \bip{|y_{\bar
        u}(t)|^{10}y_{\bar u}(t),z_k(t)}^2 \dd t.
  \end{multline}
  The limit for the first term
  \begin{equation*}
   \lim_{k\to\infty} \int_0^\sT \|y_{\bar u}(t)\|^{-8}_{\Le^{12}(\Omega)}
    \bip{|y_{\bar u}(t)|^{10},z_k^2(t)}\dd t 
   =  \int_0^\sT \|y_{\bar u}(t)\|^{-8}_{\Le^{12}(\Omega)}
    \bip{|y_{\bar u}(t)|^{10},z_v^2(t)}\dd t
  \end{equation*}
  can be proven analogously to the above with H\"older's inequality
  and the Vitali theorem due to boundedness of $(z_k)$ in $\cY_+$. For
  the second term in~\eqref{eq:psi-second-deriv-formula}, we first show that
  $\bip{|y_{\bar u}(t)|^{10}y_{\bar u}(t),z_k(t)}$ converges towards
  $\bip{|y_{\bar u}(t)|^{10}y_{\bar u}(t),z_v(t)}$ pointwise a.e.\ on
  $[0,\sT]$ as $k\to\infty$. Then convergence of the overall integral
  follows from H\"older's inequality
  \begin{multline*}
    \int_{E_t} \|y_{\bar u}(t)\|_{\Le^{12}(\Omega)}^{-20}
    \bip{|y_{\bar u}(t)|^{10}y_{\bar u}(t),z_k(t)}^2 \dd t \\
    \leq \|y_{\bar
      u}(t)\|_{\Le^4(E_t;\Le^{12}(\Omega))}^2\|z_k(t)\|_{\Le^4(0,\sT;\Le^{12}(\Omega))}^2
  \end{multline*}
  for every measurable subset $E_t \subseteq (0,\sT)$ and yet
  another application of the Vitali theorem, using boundedness
  of $(z_k)$ in $\cY_+$.

  Pointwise convergence of
  $(\bip{|y_{\bar u}(t)|^{10}y_{\bar u}(t),z_k(t)})$ can be
  obtained as follows: We had seen that $z_k(t) \to z_v(t)$ in
  $\Le^2(\Omega)$ for every $t \in [0,\sT]$. Thus for every such
  $t$ there exists a subsequence $(z_{k_m}(t))$ such that
  $z_{k_m}(t) \to z_v(t)$ almost everywhere on
  $\Omega$. Moreover, we estimate for every measurable subset
  $E_\x \subseteq \Omega$
  \begin{equation*}
    \int_{E_\x} |y_{\bar u}(t)|^{11}|z_{k_m}(t)| \dd \x \leq
    \|y_{\bar
      u}(t)\|_{\Le^{\frac{66}5}(E_\x)}^{11}\|z_{k_m}(t)\|_{\Le^6(\Omega)}
    \lesssim \|y_{\bar u}(t)\|_{\Le^{\frac{66}5}(E_\x)}^{11},
  \end{equation*}
  since $(z_{k_m})$ is bounded in
  $\cY_+ \embeds \C([0,\sT];\Le^6(\Omega))$. The general
  form of the Strichartz
  estimates~\eqref{eq:strichartz-estimate} as in~\cite{BSS09}
  shows that in fact
  $y_{\bar u} \in
  \Le^{\frac{11}3}(0,\sT;\Le^{\frac{66}5}(\Omega))$. Since $t$
  was fixed, the foregoing expression thus goes to zero
  uniformly in $m$ as $|E_\x|$ goes to zero. Hence, again the
  Vitali convergence theorem together with a
  subsequence-subsequence argument shows that
  \begin{equation*}
    \int_{\Omega} |y_{\bar u}(t)|^{10}y_{\bar u}(t)z_{k}(t) \dd
    \x \quad \xrightarrow{~k\to\infty~} \quad \int_{\Omega} |y_{\bar u}(t)|^{10}y_{\bar u}(t)z_{v}(t) \dd \x.
  \end{equation*}
  This finally implies that overall
  $\Psi''(y_{\bar u})z_k^2 \to \Psi''(y_{\bar u})z_v^2$ and
  finishes the proof.
\end{proof}


\subsubsection{First-order necessary conditions}

The \emph{tangent cone} to $\uad$ in a point
$u \in \uad$ is
\begin{multline*}
  \cT(u) \defn \Bigl\{ v \in \Le^r(0,\sT;\Le^2(\Omega)) \colon \bip{v(t),u(t)}
  \leq 0~\text{if}~\|u(t)\|_{\Le^2(\Omega)} = \omega(t) > 0, \\ v(t) =
  0~\text{if}~\|u(t)\|_{\Le^2(\Omega)} = \omega(t) = 0\Bigr\}.
\end{multline*}
We use it to state the basic first-order necessary optimality
condition for~\eqref{eq:ROCP} in a concise form. We refer to~\cite[Thm.~3.1]{CHW17}
for the routine proof.

\begin{theorem}[First-order necessary optimality condition]\label{thm:FONC}
  Let $\bar u \in \uad$ be a locally optimal solution of~\eqref{eq:ROCP}. Then there exists
  $\bar\lambda \in \partial\sj(\bar u) \subset \Le^\infty(0,\sT;\Le^2(\Omega))$ such
  that
  \begin{equation}\label{eq:FONC}
    \int_0^\sT\bip{ \bar p(t) + \beta_{1}\bar\lambda(t) +
      \beta_{2}\bar u(t), v(t)} \dd t  \geq 0 \quad \text{for
      all}~v \in \cT(\bar u).
  \end{equation}
\end{theorem}

Note that since $\sj$ is Lipschitz continuous and convex, $\bar\lambda \in
\partial\sj(\bar u)$ and~\eqref{eq:FONC} imply that (cf.~\cite[Lem.~4.2]{CHW17})
\begin{equation}
  \label{eq:FONC-direc-deriv}
  F'(\bar u)v + \beta_1\sj'(\bar u;v) \geq 0 \quad \text{for all}~v \in \cT(\bar u).
\end{equation}

We henceforth always consider a fixed locally optimal control $\bar u \in \uad$ and
the optimality condition~\eqref{eq:FONC} as given. Let us further subdivide $(0,\sT)$
into active and inactive regions w.r.t.\ the constraint in~\eqref{eq:ROCP} by
defining the following sets:
\begin{equation*}
  \cI \defn \bsetof{\|\bar u\|_{\Le^2(\Omega)} < \omega}, \quad \cA_+ \defn
  \bsetof{\|\bar u\|_{\Le^2(\Omega)} =
  \omega > 0}, \quad \cA_0 \defn \bsetof{\omega = 0}.
\end{equation*}
We first show that the integrated optimality
condition~\eqref{eq:FONC} is equivalent to the pointwise one.

\begin{corollary}
  \label{cor:FONC-pointwise}
  Condition~\eqref{eq:FONC} is equivalent to
  \begin{equation}
    \label{eq:FONC-pointwise}
    \bip{ \bar p(t) + \beta_{1}\bar\lambda(t) + \beta_{2}\bar u(t), v(t)} \geq 0 \quad
\text{for almost all}~t \in (0,\sT)
  \end{equation}
  for all $v \in \cT(\bar u)$. It moreover follows that
  \begin{equation}
    \label{eq:FONC-pointwise-inactive-zero}
    \bar p(t) + \beta_{1}\bar\lambda(t) + \beta_{2}\bar u(t) = 0
    \quad \text{for almost all}~t \in \cI.
  \end{equation}
\end{corollary}

\begin{proof}
 It is obvious that~\eqref{eq:FONC-pointwise}
 implies~\eqref{eq:FONC}. For the other way around, it suffices
 to observe that if $v \in \cT(\bar u)$, then also $\chi_N
 v \in \cT(\bar u)$ for every measurable set $N \subseteq
 (0,\sT)$,  so one can do the usual proof by
 contradiction. Quite
 similarly,~\eqref{eq:FONC-pointwise-inactive-zero} follows from
 inserting $\pm \chi_{\cI}\Le^r(0,\sT;\Le^2(\Omega)) \subset
 \cT(\bar u)$ into~\eqref{eq:FONC-pointwise}.
\end{proof}

The next result is then an observation regarding \emph{sparsity} and regularity of an optimal
control $\bar u$, as well as uniqueness of the subgradient $\bar\lambda$.
The proof is analogous
to the one in~\cite[Cor.~3.9]{CHW17} using~\eqref{eq:FONC-pointwise} and~\eqref{eq:FONC-pointwise-inactive-zero}.

\begin{corollary}\label{cor:sparsity-in-time-optimal}
  The following
  properties hold true:
  \begin{itemize}[leftmargin=*]
  \item If $\beta_2 > 0$:
    We have $\bar u \in \Le^\infty(0,\sT;\Le^2(\Omega))$ and thus $\omega \in
    \Le^\infty(\cA_+)$. Moreover, for almost all $t \in \cI$, the following
    equivalence holds true:
    \begin{equation*}\|\bar u(t)\|_{\Le^2(\Omega)} = 0 \quad \iff \quad \|\bar
      p(t)\|_{\Le^2(\Omega)} \leq \beta_1.
    \end{equation*}
  \item If $\beta_2 = 0$: For almost all $t \in \cI$, we have the implications
    \begin{equation*}
      \|\bar p(t)\|_{\Le^2(\Omega)} < \beta_1 \quad \implies \quad \|\bar
      u(t)\|_{\Le^2(\Omega)} = 0 \quad \implies \quad\|\bar p(t)\|_{\Le^2(\Omega)} \leq \beta_1.
    \end{equation*}
  \end{itemize}
  In both cases, $\bar \lambda \in \Le^\infty(0,\sT;\Le^2(\Omega))$ satisfies
  \begin{equation*}
    \bar\lambda(t) =
    \begin{cases}
      \frac{\bar u(t)}{\|\bar u(t)\|_{\Le^2(\Omega)}} & \text{if}~\|\bar
      u(t)\|_{\Le^2(\Omega)}
      \neq 0, \\[0.5em]
      -\frac1{\beta_1} \bar p(t) & \text{if}~\|\bar u(t)\|_{\Le^2(\Omega)} = 0
    \end{cases} \quad \text{f.a.a.}~t\in \cA_+ \cup \cI.
  \end{equation*}
  It is thus unique on $\cA_+ \cup \cI$.
\end{corollary}

Using Corollaries~\ref{cor:FONC-pointwise} and~\ref{cor:sparsity-in-time-optimal} we
can show that a unique bounded Lagrange multiplier associated to $\uad$ and the
locally optimal control $\bar u$ exists. We use the convention that $\frac00 = 0$.

\begin{definition}[Lagrange multiplier]
  \label{def:lagrange-multiplier}
  We say that a measurable function $\bar\mu \colon \cA_+ \cup \cI \to [0,\infty)$ is
  a \emph{Lagrange multiplier associated to $\uad$} if
  $\bar\mu(t)(\|\bar u(t)\|_{\Le^2(\Omega)} - \omega(t)) = 0$ for almost all
  $t \in \cA_+ \cup \cI$ is satisfied (complementarity), and the gradient equation
  \begin{equation}\label{eq:FONC-multiplier}
    \bar p(t) + \beta_{1}\bar\lambda(t) + \beta_{2}\bar u(t) + \bar \mu(t) \frac{\bar
      u(t)}{\|\bar u(t)\|_{\Le^2(\Omega)}} = 0 \quad \text{for almost all}~t \in \cA_+
    \cup \cI
  \end{equation}
  holds true.
\end{definition}

\begin{lemma}
  \label{lem:multiplier-existence}
  There exists a unique Lagrange multiplier
  $\bar \mu \in \Le^\infty(\cA_+ \cup \cI)$ associated to $\uad$.
\end{lemma}

\begin{proof}
  We set, of course, necessarily $\bar\mu(t) = 0$ for $t\in\cI$. Then the
  complementarity condition and~\eqref{eq:FONC-multiplier} on $\cI$ are already
  satisfied, the latter due to~\eqref{eq:FONC-pointwise-inactive-zero}.

  The next step is to show that there exists a Lagrange multiplier
  $\bar\mu \in \Le^1(\cA_+)$ associated to $\uad$. This $\bar \mu$ is then
  necessarily already an element of $\Le^\infty(\cA_+)$ and moreover unique, which we
  see as follows: From taking $\Le^2(\Omega)$ norms in~\eqref{eq:FONC-multiplier} for
  $t \in \cA_+$ it follows that
  \begin{equation*} \bigl\|\bar p(t) + \beta_{1}\bar\lambda(t) + \beta_{2}\bar
    u(t)\bigr\|_{\Le^2(\Omega)} = \bar\mu(t) \quad \text{for almost all}~t \in \cA_+.
  \end{equation*}
  The left-hand side is an $\Le^\infty(\cA_+)$ function in $t$ and unique due to
  Corollary~\ref{cor:sparsity-in-time-optimal}.

  It thus remains to show that the $\Le^1(\cA_+)$ Lagrange multiplier exists in the
  first place.  Suppose the contrary, i.e., that
  $\bar p + \beta_1 \bar\lambda + \beta_2 \bar u \neq -\mu \|\bar
  u\|_{\Le^2(\Omega)}^{-1}\bar u$ in $\Le^1(\cA_+;\Le^2(\Omega))$ for all
  $\mu \in \Le^1(\cA_+)$ with $\mu \geq 0$ a.e.. Then the Hahn-Banach theorem yields
  a function $\varphi \in \Le^\infty(\cA_+)$ such that
  \begin{equation*}
    \int_{\cA_+} \frac{\bip{\varphi(t),-\mu(t) \bar u(t)}}{\|\bar
      u(t)\|_{\Le^2(\Omega)}} \dd t \leq 0 <
    \int_{\cA_+} \bip{\varphi(t),\bar p(t) + \beta_1
      \bar\lambda(t) + \beta_2 \bar u(t)} \dd t.
  \end{equation*}
  From the first inequality it follows that
  $-\chi_{\cA_+}\varphi \in \cT(\bar u)$ (proof by contradiction) which however
  is incompatible with the second one by the first order necessary
  condition~\eqref{eq:FONC}. Hence, there exists the searched-for
  $\bar\mu \in \Le^1(\cA_+)$ satisfying $\bar \mu \geq 0$
  and~\eqref{eq:FONC-multiplier} on $\cA_+$. This finishes the proof
\end{proof}

Henceforth, $\bar \mu$ will denote the unique Lagrange multiplier associated to
$\uad$ for the locally optimal control $\bar u$.

\subsubsection{Second order necessary conditions}

We define the \emph{critical cone} $C(\bar u)$ associated to $\uad$ in a locally
optimal control $\bar u$ to consist of tangential directions along which the
directional derivative of $\ell_r$ vanishes, so
\begin{equation*}
  C(\bar u) \defn \Bigl\{v \in \cT(\bar u) \colon F'(\bar u)v + \beta_1\sj'(\bar u;v) = 0\Bigr\}.
\end{equation*}

Due to Lipschitz continuity of $\sj$, it is straightforward to show that $C(\bar u)$
is a closed convex cone in $\Le^r(0,\sT;\Le^2(\Omega))$. 

A formal computation shows that the second derivatives of $\sj$ in $u$ in directions
$(v,v)$ should be given by
\begin{equation*}
  \sj''(u;v^2) \defn \int_{\bigl[\|u\|_{\Le^2(\Omega)} \neq 0\bigr]}
  \|u(t)\|_{\Le^2(\Omega)}^{-1} \left[\|v(t)\|_{\Le^2(\Omega)}^2
    -
    \left(\frac{\bip{u(t),v(t)}}{\|u(t)\|_{\Le^2(\Omega)}}\right)^2\right]
  \dd t, 
\end{equation*}
where we consider the whole expression as $0$ if $u = 0$. Clearly, this expression is
always nonnegative, but there may be directions $v \in \Le^r(0,\sT;\Le^2(\Omega))$
for which $\sj''(u;v^2) = \infty$. In this sense, $\sj''(u;\cdot)$ should not be seen
as a traditional derivative of $\sj'$ at $u$. We still set
$\ell_r''(u;v^2) \defn F''(u)v^2 + \beta_1\sj''(u;v^2)$ which is, any way, a useful
object, as the following second order necessary conditions shows. Its proof will
occupy the rest of this subsection:

\begin{theorem}[Second order necessary conditions]
  \label{thm:sonc}
  Assume that $\omega \in \Le^1(0,\sT)$. Let $\bar u\in \uad$ be a locally optimal
  solution to~\eqref{eq:ROCP} and let $\bar \mu \in \Le^\infty(\cA_+ \cup \cI)$ be
  the associated Lagrange multiplier. Then there holds
  \begin{equation*}
    \ell_r''(\bar u;v^2) + \int_{\cA_+} \bar\mu(t)  \bigl\|\bar
    u(t)\bigr\|_{\Le^2(\Omega)}^{-1} \left[\|v(t)\|_{\Le^2(\Omega)}^2 - 
      \left(\frac{\bip{\bar u(t),v(t)}}{\|\bar u(t)\|_{\Le^2(\Omega)}}\right)^2\right]
    \dd t \geq 0
  \end{equation*}
  for all $v \in C(\bar u)$.
\end{theorem}

The expression in Theorem~\ref{thm:sonc}
corresponds to $\partial^2_u L(\bar u,\bar \mu;v^2) \geq 0$ with the Lagrangian
\begin{equation*}
L(u,\mu) \defn \ell_r(u) + \int_0^\sT \mu(t)\bigl(\|u(t)\|_{\Le^2(\Omega)} -
\omega(t)\bigr) \dd t,
\end{equation*}
where in the theorem we have already inserted $\bar\mu = 0$ a.e.\ on $\cI$. Both
$\sj''(\bar u;v^2)$ and the explicit integral in the substitute for the second
derivative of the Lagrange penalty term in Theorem~\ref{thm:sonc} may be infinite. We
emphasize once more that we do not require $\omega$ to be bounded away from zero. In
the case $\bar u \equiv 0$, the condition in Theorem~\ref{thm:sonc} collapses to
$F''(0)v^2 \geq 0$ for all $v \in C(0)$.

\begin{remark}
  \label{rem:lagrange-form}
  It is also possible to obtain an analogous result for the Lagrangian with a
  quadratic penalty term
  \begin{equation*}
    L_2(u,\mu) \defn \ell_r(u) + \int_0^\sT \mu(t)\bigl(\|u(t)\|_{\Le^2(\Omega)}^2 -
    \omega(t)^2\bigr) \dd t.
  \end{equation*}
  Then the necessary condition in Theorem~\ref{thm:sonc} becomes
  \begin{equation*}
    \ell_r''(\bar u_2;v^2) + \int_{\cA_+} \bar\mu(t)
    \frac{\|v(t)\|_{\Le^2(\Omega)}^2}{\|\bar u(t)\|_{\Le^2(\Omega)}}
    \dd t \geq 0
  \end{equation*}
  with the same multiplier $\bar \mu$ as before.  The integral in the foregoing
  expression may also be infinite. The proof works nearly exactly as the one for
  Theorem~\ref{thm:sonc} presented below.
\end{remark}

We next prepare for the proof of Theorem~\ref{thm:sonc} with some auxiliary
results. From~\cite[Prop.~4.1/\allowbreak{}Lem.~4.2]{CHW17} together with the
pointwise first-order necessary condition~\eqref{eq:FONC-pointwise} and the Lagrange
gradient equation~\eqref{eq:FONC-multiplier} we obtain the first lemma for critical
directions:

\begin{lemma}\label{lem:critical-direct-prop}
  For all $v \in C(\bar u)$, there holds
  \begin{equation*}
    \sj'(\bar u;v) = \int_0^\sT \bip{\bar\lambda(t),v(t)}\dd t
  \end{equation*}
  and thus
    \begin{equation}
      0 = \bip{\bar p(t) + \beta_1 \bar\lambda(t) +
        \beta_2 \bar u(t),v(t)} = -
      \bar\mu(t) \frac{\bip{\bar u(t),v(t)}}{\|\bar u(t)\|_{\Le^2(\Omega)}} 
    \label{eq:critical-gradient-zero}
  \end{equation}
  for almost all $t \in \cA_+ \cup \cI$.
\end{lemma}

Equation~\eqref{eq:critical-gradient-zero} also shows that if $t \in \cA_+$ and
$\ip{\bar u(t),v(t)} < 0$ for some $v \in C(\bar u)$, then $\bar\mu(t) = 0$
follows.

In the proof of Theorem~\ref{thm:sonc}, we will need properties of
$\ell_r'(\bar u;w)$ with directions $w$ which are possibly not in the critical cone,
but derived from some $v \in C(\bar u)$. The next lemma gives the required results.

\begin{lemma}
  \label{lem:critical-direct-multiple-prop}
  Let $v \in C(\bar u)$ be given and let $w \in \Le^r(0,\sT;\Le^2(\Omega))$ be
  another function such that for almost all
  $t \in \setof{\|\bar u\|_{\Le^2(\Omega)} = 0}$, $w(t)$ is either $v(t)$ or
  zero. Then
  \begin{equation*}
    F'(\bar u)w + \beta_1\sj'(\bar u;w) = -\int_{\cA_+} \bar \mu(t) \frac{\bip{\bar
        u(t),w(t)}}{\|\bar u(t)\|_{\Le^2(\Omega)}} \dd t.
  \end{equation*}
  Further, there holds $\chi_MC(\bar u) \subseteq C(\bar u)$ for any measurable set
  $M \subseteq (0,\sT)$.
\end{lemma}

\begin{proof}
  Let $v \in C(\bar u)$. Arguing as for~\cite[(4.12)]{CHW17}, we obtain that
  \begin{equation}
    \label{eq:lambda-critical-ident}
    \bar \lambda(t) = \frac{v(t)}{\|v(t)\|_{\Le^2(\Omega)}} \quad \text{f.a.a.}~t \in
    \bsetof{\|\bar u\|_{\Le^2(\Omega)} = 0} \cap   \bsetof{\|v\|_{\Le^2(\Omega)} \neq 0}.
  \end{equation}
  Now let $w(t)$ be either $v(t)$ or zero for almost all
  $t \in \setof{\|\bar u\|_{\Le^2(\Omega)} = 0}$. Then
  \begin{equation*}
    \bip{v(t),w(t)} = \|v(t)\|_{\Le^2(\Omega)}\|w(t)\|_{\Le^2(\Omega)} \quad
    \text{f.a.a.}~t \in \bsetof{\|\bar u\|_{\Le^2(\Omega)} = 0}. 
  \end{equation*}
  Using this together with~\eqref{eq:lambda-critical-ident}, we find
  \begin{equation*}
    \int_{\setof{\|\bar u\|_{\Le^2(\Omega)} = 0}} \bip{\bar\lambda(t),w(t)} \dd t =
    \int_{\setof{\|\bar u\|_{\Le^2(\Omega)} = 0}} \|w(t)\|_{\Le^2(\Omega)} \dd t
  \end{equation*}
  and thus, with~\eqref{eq:FONC-pointwise-inactive-zero} and~\eqref{eq:FONC-multiplier},
  \begin{align}
    F'(\bar u)w + \beta_1 \sj'(\bar u;w) & = \int_0^\sT \bip{\bar p(t) + \beta_2 \bar
      u(t),w(t)} \dd t + \beta_1 \int_{\setof{\|\bar u\|_{\Le^2(\Omega)} = 0}}
    \|w(t)\|_{\Le^2(\Omega)}\dd t\notag\\ & \qquad  + 
    \beta_1 \int_{\setof{\|\bar u\|_{\Le^2(\Omega)}\neq0 }}
    \bip{\bar\lambda(t),w(t)}\dd t\notag 
    \\ &  = 
    \int_{\cA_+} \bip{\bar p(t) + \beta_1 \bar\lambda(t) + \beta_2 \bar
      u(t),w(t)}\dd t \notag\\  & =  - \int_{\cA_+} \bar
    \mu(t)  \frac{\bip{\bar
        u(t),w(t)}}{\|\bar u(t)\|_{\Le^2(\Omega)}}\dd t.\label{eq:parallelformula} 
  \end{align}
  (See Proposition~\ref{prop:L1L2-subdiff-character} for the derivative formula for
  $\sj'(\bar u;w)$.) This was the first claim.

  Let now $w = \chi_M v$ for some measurable set
  $M \subseteq (0,\sT)$. Then~\eqref{eq:parallelformula} holds true. Moreover,
  $\bar\mu(t) \|\bar u(t)\|_{\Le^2(\Omega)}^{-1}\ip{\bar u(t),w(t)} = 0$ for almost
  all $t \in \cA_+$ by~\eqref{eq:lambda-critical-ident} which again
  in~\eqref{eq:parallelformula} shows that
  $F'(\bar u)w + \beta_1 \sj'(\bar u;w) = 0$, so $w \in C(\bar u)$.
\end{proof}

We further want to use second-order Taylor approximations for $\sj$. These are not
immediate since we have already seen that the substitute for the second order
derivative $\sj''(\bar u;v^2)$ may be infinite for some directions $v$.

Consider $\Upsilon_2(f) := \|f\|_{\Le^2(\Omega)}$. We have
\begin{align*}
  \Upsilon_2'(f)h & = \|f\|^{-1}_{\Le^2(\Omega)} \bip{f,h} ,\\
   \Upsilon_2''(f)h^2 & = 
   \|f\|_{\Le^2(\Omega)}^{-1}\|h\|_{\Le^2(\Omega)}^2 -
   \|f\|_{\Le^2(\Omega)}^{-3} \bip{f,h}^2
\intertext{for $f,h \in \Le^2(\Omega)$ with $f \neq 0$, cf.\ the appendix, and we will also need}
  \Upsilon_2'''(f)h^3 & = 3\,\|f\|^{-3}_{\Le^2(\Omega)}
  \Bigl[\|f\|^{-2}_{\Le^2(\Omega)}\bip{f,h}^3 - 
  \|h\|_{\Le^2(\Omega)}^2 \bip{f,h}\Bigr]
\end{align*}
now; this is obtained by the chain rule
since $\Upsilon_2''$ is composed of continuously differentiable functions away from zero.

\begin{lemma}
  \label{lem:taylor-expansion-norm}
  Let $M \subseteq (0,\sT)$ be a measurable set.
  \begin{enumerate}[leftmargin=*]
  \item\label{item:taylor-expansion-norm-ineq} Let $f,h \in \Le^1(M;\Le^2(\Omega))$. Then
    \begin{equation*}
      \int_M \Bigl(\Upsilon_2\bigl(f(t)+h(t)\bigr) - \Upsilon_2(f(t))\Bigr) \dd t \geq \int_M
      \Upsilon_2'(f(t))h(t)  \dd t.
    \end{equation*}
  \item\label{item:taylor-expansion-norm-quadratic} Let moreover
    $\eta \in \Le^\infty(M)$. Suppose that there is a number $\alpha > 0$ such that
    $\|f(t)\|_{\Le^2(\Omega)} \geq \alpha$ for almost all $t \in M$. Then
    $h \mapsto \int_M \eta(t)) \Upsilon_2''(f(t))h(t)^2 \dd t$ defines a continuous
    quadratic form on $\Le^2(M;\Le^2(\Omega))$. If $\eta \geq 0$ a.e.\ on $M$, then
    the quadratic form is convex.
  \item\label{item:taylor-expansion-norm-integrated} Consider further
    $h \in \Le^3(M;\Le^2(\Omega))$. If for all functions $\theta \colon M \to [0,1]$
    there is $\alpha_\theta > 0$ such that
    $\|f(t) + \theta(t)h(t)\|_{\Le^2(\Omega)} \geq \alpha_\theta$ for almost all
    $t \in M$, then we have the Taylor expansion
    \begin{multline}
      \int_M \eta(t)\Bigl(\Upsilon_2\bigl(f(t)+h(t)\bigr) - \Upsilon_2\bigl(f(t)\bigr)\Bigr) \dd
      t \\ = \int_M \eta(t)\Bigl(\Upsilon_2'\bigl(f(t)\bigr)h(t) + \frac12
      \Upsilon_2''\bigl(f(t)\bigr)h(t)^2\Bigr) \dd t +
      \cO\bigl(\|h\|_{\Le^3(M;\Le^2(\Omega))}^{3}\bigr).\label{eq:norm-taylor}
    \end{multline}
  \end{enumerate}
\end{lemma}

\begin{proof}
  \begin{enumerate}[wide, labelwidth=!, labelindent=0pt]
  \item We consider the Taylor expansion for
    $\Upsilon_2(f(t)+ h(t))$ for almost every $t \in M$, with a function
    $\vartheta \colon M \to [0,1]$:
    \begin{equation*}
      \Upsilon_2\bigl(f(t)+h(t)\bigr) - \Upsilon_2\bigl(f(t)\bigr) = \Upsilon_2'\bigl(f(t)\bigr)h(t) +
      \frac12 \Upsilon_2''\bigl(f(t) + \vartheta(t)h(t)\bigr)h(t)^2.
    \end{equation*}
    Since $\Upsilon_2''(g)w^2 \geq 0$ for all $g,w \in \Le^2(\Omega)$, the claim follows
    from inserting this in the foregoing inequality and integrating over $M$. The
    integrals are finite due to $f,h \in \Le^1(M;\Le^2(\Omega))$.
  \item Under the assumptions on $f$, we find
    \begin{equation*}
      \int_M \bigl|\eta(t)\Upsilon_2''(f(t))h(t)^2\bigr| \dd t \leq 2 \alpha^{-1}
      \|\eta\|_{\Le^\infty(M)} 
      \int_M\bigl\|h(t)\bigr\|_{\Le^2(\Omega)}^2 \dd t.
    \end{equation*}
    This implies the continuity assertion. Moreover, a quadratic form is convex if
    and only if it is nonnegative, and the latter is ensured by $\eta \geq 0$ a.e.\
    on $M$.
  \item For the Taylor expansion for the integrated $\Upsilon_2$, we again have from
    Taylor expansion for $\Upsilon_2(f(t)+ h(t))$ for almost every $t \in M$, with a
    function $\vartheta \colon M \to [0,1]$:
    \begin{multline*}
      \Upsilon_2\bigl(f(t)+h(t)\bigr) - \Upsilon_2\bigl(f(t)\bigr) \\ =
      \Upsilon_2'\bigl(f(t)\bigr)h(t) + \frac12 \Upsilon_2''\bigl(f(t)\bigr)h(t)^2 + \frac16
      \Upsilon_2'''\bigl(f(t) + \vartheta(t)h(t)\bigr)h(t)^3.
    \end{multline*}
    If
    $\bigl\|f(t) + \vartheta(t)h(t)\bigr\|_{\Le^2(\Omega)} \geq \alpha_\vartheta > 0$
    for almost all $t \in M$, then
    \begin{multline*}
      \int_M \bigl|\eta(t)\Upsilon_2'''\bigl(f(t) + \vartheta(t)h(t)\bigr)h(t)^3\bigr|
      \dd t \\ \leq 6 \alpha_\vartheta^{-2} \|\eta\|_{\Le^\infty(M)} \int_M
      \bigl\|h(t)\bigr\|_{\Le^2(\Omega)}^3 \dd t \in
      \cO\bigl(\|h\|_{\Le^3(M;\Le^2(\Omega))}^{3}\bigr).
    \end{multline*}
    The claim thus follows from multiplying the Taylor expansion for
    $\Upsilon_2(f(t) + h(t))$ by $\eta(t)$ and integrating over $M$. \qedhere
  \end{enumerate}
\end{proof}

We next give the proof of Theorem~\ref{thm:sonc}. The principal idea is to
approximate the critical direction $v \in C(\bar u)$ in multiple stages.

\begin{proof}[Proof of Theorem~\ref{thm:sonc}] Let $v \in C(\bar u)$. 
The proof is achieved as follows: We first suppose that $v \in  \Le^\infty(0,\sT;\Le^2(\Omega))$
and that
\begin{equation}
  \int_{\setof{\|\bar u\|_{\Le^2(\Omega)} \neq 0}} \frac{\|v(t)\|_{\Le^2(\Omega)}^2}{\|\bar
    u(t)\|_{\Le^2(\Omega)}} < \infty.\label{eq:prelim-assu}
\end{equation}
Since a multiple of this integral is an upper bound for
$\sj''(\bar u;v^2)$,~\eqref{eq:prelim-assu} implies that $\sj''(\bar u;v^2)$ is
finite. We then construct two-staged approximations $u_{\rho,k}$ of $\bar u$ such
that $u_{\rho,k} \to \bar u$ uniformly as $\rho \searrow 0$, as well as
$u_{\rho,k} \in \uad$ for $\rho > 0$ small enough and $k$ fixed. Another property we
need later is that $\|u_{\rho,k}(t)\|_{\Le^2(\Omega)} = \omega(t)$ for $t \in \cA_+$
with $\ip{\bar u(t),v(t)} = 0$. Since such constructed $u_{\rho,k}$ is feasible and
close to $\bar u$ for $\rho > 0$ sufficiently small, we then make the ansatz
$0 \leq L(u_{\rho,k},\bar\mu) - L(\bar u,\bar\mu)$ and pass to the limit in
$\rho$ and $k$ in the second order Taylor expansions there which gives the claim. For
this, we need and establish that the derivative
$v_k = \lim_{\rho\searrow 0} \rho^{-1}(u_{\rho,k} - \bar u)$ exists in
$\Le^\infty(0,\sT;\Le^2(\Omega))$ and satisfies $v_k \to v$ in
$\Le^r(0,\sT;\Le^2(\Omega))$ as $k \to \infty$. Finally, we remove the assumptions on
$v$ above.

\textbf{Step~1}: \emph{Construction of $u_{\rho,k}$.} Let $\alpha_k,\omega_k$ be
arbitrary positive sequences converging monotonically to zero. We define
\begin{multline*}
  N_k := \Bigl\{ t \in (0,\sT) \colon 0 < \|\bar u(t)\|_{\Le^2(\Omega)} <
  \alpha_k~~\text{or}~~(1-\alpha_k)\omega(t) < \|\bar u(t)\|_{\Le^2(\Omega)} < \omega(t)
  \\\text{or}~~0 \leq \omega(t) < \omega_k~~\text{or}~~\omega(t) >
  \omega_k^{-1}
  \Bigr\}.
\end{multline*}
Note that $\cA_0 \subset N_k$ and $|N_k \setminus \cA_0| \to 0$ as $k \to
\infty$. For $t \notin N_k$, we moreover have $\omega_k \leq \omega(t) \leq
\omega_k^{-1}$.
 Set
  \begin{equation*}
    u_{\rho,k}(t) =
    \begin{cases}
      \bar u(t) & \text{if}~t \in N_k,
      \\ \bigl(1-\omega_t^k(\rho)\bigr)\bar u(t) + \rho v(t)&\text{if}~t\in \cA_+
      \cap N_k^c~\text{and}~\ip{u(t),v(t)} = 0, \\
      (1-\rho\alpha_k)\bar u(t) + \rho v(t) &\text{if}~t \in \cA_+ \cap
      N_k^c~\text{and}~\ip{u(t),v(t)} < 0, \\\bar u(t) + \rho v(t)
      &\text{elsewhere},
    \end{cases}
  \end{equation*}
  with
  \begin{equation*}
    \omega_t^k(\rho) := 1- \sqrt{1-\frac{\rho^2
        \|v(t)\|_{\Le^2(\Omega)}^2}{\omega(t)^2}} \quad 
    \text{where}~|\rho|<\tfrac12 \omega_k\|v\|_{\Le^\infty(0,\sT;\Le^2(\Omega))}^{-1}.
  \end{equation*}
  The function $\omega_t^k$ is chosen exactly such that
  $\|u_{\rho,k}(t)\|_{\Le^2(\Omega)} = \omega(t)$ for $t \in \cA_+ \cap N_k^c$ with
  $\ip{\bar u(t),v(t)} = 0$.

  We have $u_{\rho,k} \in \uad$ for $k$ fixed and $\rho$ sufficiently small as we
  observe as follows:
\begin{itemize}[leftmargin=1.5em]
\item if $t \in N_k$, then $u_{\rho,k}(t) = \bar u(t)$ which is feasible,
\item if $t \notin N_k$ and $\|\bar u(t)\|_{\Le^2(\Omega)} \leq
  (1-\alpha_k)\omega(t)$, then for $\rho \leq \alpha_k\omega_k\|v\|_{\Le^\infty(0,\sT;\Le^2(\Omega))}^{-1}$:
  \begin{equation*}
    \bigl\|u_{\rho,k}(t)\bigr\|_{\Le^2(\Omega)} = \bigl\|\bar u(t) + \rho v(t)\bigr\|_{\Le^2(\Omega)} \leq
    (1-\alpha_k)\omega(t) +  \rho \|v(t)\|_{\Le^2(\Omega)} \leq \omega(t),
  \end{equation*}
\item if $t \in \cA_+ \cap N_k^c$ and $\ip{u(t),v(t)} = 0$, then
  $\|u_{\rho,k}(t)\|_{\Le^2(\Omega)} = \omega(t)$,
\item if $t \in \cA_+ \cap N_k^c$ and $\ip{u(t),v(t)} < 0$, then
  \begin{align*}
    \bigl\|u_{\rho,k}(t)\bigr\|_{\Le^2(\Omega)}^2 \leq \omega(t)^2 \quad  & \iff \quad
    \bigl(1-\rho\alpha_k\bigr)^2\omega(t)^2 + \rho^2 \|v(t)\|_{\Le^2(\Omega)}^2 \leq
    \omega(t)^2 \\ 
    &\iff \quad \bigl(-2\alpha_k + \rho\alpha_k^2\bigr)\omega(t)^2 + \rho
    \|v(t)\|^2_2 \leq 0
  \end{align*}
  and the latter is satisfied uniformly in $t$ for this case if
  \begin{equation*}
    \rho \leq \frac{2\alpha_k\omega_k}{\alpha_k^2 \omega_k^{-2} +
      \|v\|_{\Le^\infty(0,\sT;\Le^2(\Omega))}}.
  \end{equation*}
\end{itemize}

\textbf{Step~2}: \emph{Limits as $\rho \searrow 0$ and $k \to \infty$}.
It is clear that $u_{\rho,k}(t) \to \bar u(t)$ in $\Le^2(\Omega)$ as
$\rho \searrow 0$ for almost every $t \in (0,\sT)$. We show that this convergence is
in fact uniform. First, note that due to
$\|\bar u(t)\|_{\Le^2(\Omega)} = \omega(t) \leq \omega_k^{-1}$ on $\cA_+ \cap N_k^c$,
we have $\bar u \in \Le^\infty(\cA_+ \cap N_k^c;\Le^2(\Omega))$. It follows that
$w_{\rho,k} \defn u_{\rho,k}- \bar u \in \Le^\infty(0,\sT;\Le^2(\Omega))$. For
$t \in \cA_+ \cap N_k^c$, we have
\begin{equation}\label{eq:gamma-uniform-bound}
  0 \leq \omega_t^k(\rho) \leq 1 -
  \sqrt{1-\frac{\rho^2\|v\|_{\Le^\infty(0,\sT;\Le^2(\Omega))}^2}{\omega_k^2}} \leq
  \frac{\rho^2\|v\|_{\Le^\infty(0,\sT;\Le^2(\Omega))}^2}{\omega_k^2},
\end{equation}
so $\omega_t^k(\rho)$ tends to zero uniformly in $t \in \cA_+ \cap N_k^c$ as
$\rho \searrow 0$ and hence
\begin{equation}\label{eq:urhok-Op}
  \bigl\|w_{\rho,k}\bigr\|_{\Le^\infty(0,\sT;\Le^2(\Omega))} = \bigl\|u_{\rho,k}-\bar
  u\bigr\|_{\Le^\infty(0,\sT;\Le^2(\Omega))} \in \cO(\rho). 
\end{equation}

Next, set
\begin{equation*}
  v_k(t) := \begin{cases}
    0 & \text{if}~t \in N_k, \\
    v(t) -
    \alpha_k \bar u(t) &\text{if}~t\in\cA_+ \cap N_k^c~\text{and}~\ip{\bar
      u(t),v(t)} < 0,  \\ v(t) 
    &\text{elsewhere}.
  \end{cases}
\end{equation*}
Then
\begin{equation}
v_k \to v \quad \text{in}~\Le^r\bigl(0,\sT;\Le^{2}(\Omega)\bigr)  \quad\text{as}~k \to \infty.\label{eq:vk-conv}
\end{equation}
Further,
$\omega_t^k(\cdot)$ is continuously differentiable with
$(\omega_t^k)'(0) = 0 = \omega_t^k(0)$. Thus
\begin{equation*}
  v_k(t) = \lim_{\rho\searrow 0}
  \rho^{-1} w_{\rho,k}(t)
  \quad \text{for almost every}~t \in
  (0,\sT). 
\end{equation*}
We again show that this convergence is uniform. In fact, for
$t \in \cA_+ \cap N_k^c$, we have
$\omega_t^k(\rho) = \omega_t^k(0) + (\omega_t^k)'(0)\rho + \omega_t^k(\rho)$ and,
via~\eqref{eq:gamma-uniform-bound},
\begin{equation*}
  0 \leq \frac{\omega_t^k(\rho)-\omega_t^k(0)}{\rho} =\frac{\omega_t^k(\rho)}{\rho}
  \leq  \frac{\rho\|v\|_{\Le^\infty(0,\sT;\Le^2(\Omega))}}{\omega_k^2}.
\end{equation*}
Hence the difference quotients for
$(\omega_t^k)'(0) = \lim_{\rho \searrow 0} \rho^{-1}(\omega^k_t(\rho)-\omega^k_t(0))$
converge uniformly in $t \in \cA_+ \cap N_k^c$.  We obtain that
\begin{equation}
  v_k = \lim_{\rho\searrow 0} \rho^{-1}w_{\rho,k}
  \quad  \text{in}~\Le^\infty(0,\sT;\Le^2(\Omega)).\label{eq:rho-vk-conv}
\end{equation}

\textbf{Step 3}: \emph{Core of the proof}.
We first check that, for $t \in \cA_+ $,
\begin{equation*}
\bar\mu(t)\bigl(\|u_{\rho,k}(t)\|_{\Le^2(\Omega)} - \|\bar u(t)\|_{\Le^2(\Omega)}\bigr) = 0.
\end{equation*}
This is true because of the following:
\begin{itemize}[leftmargin=1.5em]
\item for $t \in \cA_+ \cap N_k$, we have $u_{\rho,k}(t) = \bar u(t)$,
\item for $t \in \cA_+ \cap N_k^c $ with $\ip{u(t),v(t)} = 0$, we have
  $\|u_{\rho,k}(t)\|_{\Le^2(\Omega)} = \|\bar u(t)\|_{\Le^2(\Omega)}$ by construction,
\item and for $t \in \cA_+ \cap N_k^c$ with $\ip{u(t),v(t)} < 0$ we had already
  seen that $\bar \mu(t) = 0$ follows from~\eqref{eq:critical-gradient-zero}.
\end{itemize}
For $\rho$ small enough and fixed $k$, we have $u_{\rho,k} \in \uad$ and
$\ell_r(\bar u) \leq \ell_r(u_{\rho,k})$ due to local optimality of $\bar u$
and~\eqref{eq:urhok-Op}. We thus make the ansatz
\begin{equation}\label{eq:nsoc-ansatz}
  0 \leq \ell_r(u_{\rho,k}) - \ell_r(\bar u) + \int_{\cA_+ } \bar
  \mu(t)\bigl(\|u_{\rho,k}(t)\|_{\Le^2(\Omega)} - \|\bar u(t)\|_{\Le^2(\Omega)}\bigr)
  \dd t.
\end{equation}
We want to employ Taylor expansions for $\ell_r = F + \beta_1 \sj$ and the multiplier
term. The direction will be $w_{\rho,k} := u_{\rho,k}-\bar u$. For $F$, this is
easily done since $F$ is twice continuously differentiable on
$\Le^r(0,\sT;\Le^2(\Omega))$ by Lemma~\ref{lem:f-c2}, but the nonsmooth terms require
some justification in order to use Lemma~\ref{lem:taylor-expansion-norm}. For both
the reference point is $\bar u$. Consider
\begin{multline}
  \sj(u_{\rho,k}) - \sj(\bar u) = \int_{[\|\bar u\|_{\Le^2(\Omega)} = 0]}
  \|u_{\rho,k}(t)\|_{\Le^2(\Omega)} \dd t \\ + \int_{[\|\bar u\|_{\Le^2(\Omega)} \neq
    0]} \bigl(\|u_{\rho,k}(t)\|_{\Le^2(\Omega)} - \|\bar u(t)\|_{\Le^2(\Omega)}\bigr)
  \dd t.\label{eq:nonsmooth-taylor-ansatz}
\end{multline}
We focus on the second integral. By construction, $u_{\rho,k}(t) = \bar u(t)$ if
$t \in \setof{0 < \|\bar u\|_{\Le^2(\Omega)} < \alpha_k}$. Hence
\begin{multline*}
  \int_{\setof{\|\bar u\|_{\Le^2(\Omega)} \neq 0}}
  \bigl(\|u_{\rho,k}(t)\|_{\Le^2(\Omega)} - \|\bar u(t)\|_{\Le^2(\Omega)}\bigr) \dd t
  \\ = \int_{\setof{\|\bar u\|_{\Le^2(\Omega)} \geq \alpha_k}}
  \bigl(\|u_{\rho,k}(t)\|_{\Le^2(\Omega)} - \|\bar u(t)\|_{\Le^2(\Omega)}\bigr) \dd
  t.
\end{multline*}
Let $t \in \setof{\|\bar u\|_{\Le^2(\Omega)} \geq \alpha_k}$ and let
$\theta_{\rho,k}(t) \in [0,1]$. Then, for $\rho$ sufficiently small we estimate
\begin{align*}
  \bigl\| u_{\rho,k}(t) + \theta_{\rho,k}(t)w_{\rho,k}(t)\bigr\|_{\Le^2(\Omega)} & \geq
  \|\bar u(t)\|_{\Le^2(\Omega)} -  
  \theta_{\rho,k}(t)\|w_{\rho,k}(t)\|_{\Le^2(\Omega)} \\ &\geq \alpha_k -
  \|w_{\rho,k}(t)\|_{\Le^2(\Omega)} \geq \frac{\alpha_k}2,
\end{align*}
since we had $\|w_{\rho,k}\|_{\Le^\infty(0,\sT;\Le^2(\Omega))} \in \cO(\rho)$,
cf.~\eqref{eq:urhok-Op}. Thus the prerequisites for the Taylor expansion in
Lemma~\ref{lem:taylor-expansion-norm} are satisfied and we obtain, using~\eqref{eq:urhok-Op},
\begin{multline*}
  \int_{\setof{\|\bar u\|_{\Le^2(\Omega)} \neq 0}}
  \bigl(\|u_{\rho,k}(t)\|_{\Le^2(\Omega)} - \|\bar u(t)\|_{\Le^2(\Omega)}\bigr)\dd t
  = \int_{\setof{\|\bar u\|_{\Le^2(\Omega)} \neq 0}} \frac{\bip{\bar
      u(t),w_{\rho,k}(t)}}{ \|\bar u(t)\|_{\Le^2(\Omega)}} \dd t \\ + \frac12
  \int_{\setof{\|\bar u\|_{\Le^2(\Omega)} \neq 0}} \|\bar u(t)\|_{\Le^2(\Omega)}^{-1}
  \left[\bigl\|w_{\rho,k}(t)\bigr\|_{\Le^2(\Omega)}^2 - \left( \frac{\bip{\bar u(t)\,
          w_{\rho,k}(t)}}{\| \bar u(t)\|_{\Le^2(\Omega)}}\right)^2\right]\dd t +
  \cO\bigl(\rho^3\bigr).
\end{multline*}
Re-inserting into~\eqref{eq:nonsmooth-taylor-ansatz}, we finally get
\begin{equation*}
  \sj(u_{\rho,k}) - \sj(\bar u) = \sj'(\bar u;w_{\rho,k}) + \sj''(\bar u;w_{\rho,k}^2) + \cO(\rho^3).
\end{equation*}
For the multiplier term, we argue analogously (thereby using
$\mu \in \Le^\infty(\cA_+)$ for Lemma~\ref{lem:taylor-expansion-norm}) to show that
\begin{multline*}
  \int_{\cA_+ } \bar \mu(t)\bigl(\|u_{\rho,k}(t)\|_{\Le^2(\Omega)} - \|\bar
  u(t)\|_{\Le^2(\Omega)}\bigr) \dd t = \int_{\cA_+ } \bar\mu(t) \frac{\bip{\bar
      u(t),w_{\rho,k}(t)}}{\|\bar u(t)\|_{\Le^2(\Omega)}} \dd t\\ + \frac12
  \int_{\cA_+ } \bar\mu(t) \|\bar u(t)\|_{\Le^2(\Omega)}^{-1} \left[
    \bigl\|w_{\rho,k}(t)\bigr\|_{\Le^2(\Omega)}^2 - \left(\frac{\bip{\bar u(t), w_{\rho,k}(t)}}{\|
        \bar u(t)\|_{\Le^2(\Omega)}}\right)^2\right] \dd t + \cO\bigl(\rho^3\bigr).
\end{multline*}
We thus obtain from the ansatz~\eqref{eq:nsoc-ansatz}, with some function
$\vartheta_{\rho,k} \colon (0,\sT) \to [0,1]$:
\begin{align*}
  0 & \leq F'(\bar
  u)w_{\rho,k} + \beta_1 \sj'(\bar u;w_{\rho,k}) + \int_{\cA_+} \bar \mu(t) \frac{\bip{\bar u(t),w_{\rho,k}(t)}}{\|\bar
  u(t)\|_{\Le^2(\Omega)}} \dd t\\ & \quad +
  \frac12 \Bigl( F''\bigl(\bar u 
  + \vartheta_{\rho,k}w_{\rho,k}\bigr)w_{\rho,k}^2   +
  \beta_1 \sj''(\bar u;w_{\rho,k}^2)\Bigr) \\ & \quad +
  \frac12 \int_{\cA_+} \bar\mu(t) \|\bar u(t)\|_{\Le^2(\Omega)}^{-1} \left[\bigl\|w_{\rho,k}(t)\bigr\|_{\Le^2(\Omega)}^2 - \left(\frac{\bip{\bar
        u(t), w_{\rho,k}(t)}}{\| \bar
        u(t)\|_{\Le^2(\Omega)}}\right)^2\right] \dd t 
  +  \cO\bigl(\rho^3\bigr).
\end{align*}
But for $t \in \setof{\|\bar u\|_{\Le^2(\Omega)}=0}$, $w_{\rho,k}(t)$ equals either
$v(t)$ or $0$, hence we have seen in~\eqref{eq:parallelformula} that
\begin{equation*}
  F'(\bar
  u)(w_{\rho,k}) + \beta_1 \sj'(\bar u;w_{\rho,k}) = -\int_{\cA_+}
  \bar\mu(t) \frac{\bip{\bar u(t),w_{\rho,k}(t)}}{\|\bar u(t)\|_{\Le^2(\Omega)}} \dd t.
\end{equation*}
Inserting in the foregoing inequality and dividing by $\rho^2/2$, we find
\begin{align*}
  0 & \leq F''\bigl(\bar u
  + \vartheta_{\rho,k}w_{\rho,k}\bigr)v_{\rho,k}^2   +
  \beta_1 \sj''(\bar u;v_{\rho,k}^2) \\ & \qquad +
  \int_{\cA_+} \bar\mu(t) \|\bar u(t)\|_{\Le^2(\Omega)}^{-1} \left[
    \bigl\|v_{\rho,k}(t)\bigr\|_{\Le^2(\Omega)}^2 - \left(\frac{\bip{\bar
          u(t), v_{\rho,k}(t)}}{\| \bar
        u(t)\|_{\Le^2(\Omega)}}\right)^2\right] \dd t+
  \cO\bigl(\rho\bigr),
\end{align*}
where we have set $v_{\rho,k}(t) := \rho^{-1} w_{\rho,k}(t)$.

We let $\rho \searrow0$. Recall that $\lim_{\rho\searrow 0} w_{\rho,k} = 0$ and
$\lim_{\rho\searrow0} v_{\rho,k} = v_k$, both in $\Le^\infty(0,\sT;\Le^2(\Omega))$
by~\eqref{eq:rho-vk-conv} and~\eqref{eq:vk-conv}. It is thus immediate from
Lemma~\ref{lem:f-c2} that
\begin{equation*}
  F''(\bar u + \vartheta_{\rho,k}w_{\rho,k})v_{\rho,k}^2 \to F''(\bar u)v_k^2 \quad \text{as}~\rho \searrow 0.
\end{equation*}
For the two other terms, we use again that if
$t \in \setof{0 < \|\bar u\|_{\Le^2(\Omega)} < \alpha_k}$, then by construction
$v_{\rho,k}(t) = 0 = v_k(t)$. Hence it suffices to
consider the integrals on $\setof{\|\bar u\|_{\Le^2(\Omega)} \geq
  \alpha_k}$. Lemma~\ref{lem:taylor-expansion-norm}~\eqref{item:taylor-expansion-norm-quadratic}
shows that the substitutes for the second derivative induce continuous quadratic
forms on $\Le^2(\setof{\|\bar u\|_{\Le^2(\Omega)} \geq \alpha_k};\Le^2(\Omega))$ and
$\Le^2(\setof{\|\bar u\|_{\Le^2(\Omega)} \geq \alpha_k}\cap \cA_+;\Le^2(\Omega))$,
respectively. Using $\lim_{\rho\searrow0} v_{\rho,k} = v_k$, we thus
obtain 
\begin{multline}
    0  \leq F''(\bar u)v_k^2 + \beta_1 \sj''(\bar u;v_k^2) \\ + \int_{\cA_+}
    \bar\mu(t) \|\bar u(t)\|_{\Le^2(\Omega)}^{-1} \left[\|v_k(t)\|_{\Le^2(\Omega)}^2 -
      \left(\frac{\bip{\bar u(t), v_k(t)}}{\| \bar
          u(t)\|_{\Le^2(\Omega)}}\right)^2\right] \dd t.
  \label{eq:nsoc-k}
\end{multline}
  
We next pass to the limit in~\eqref{eq:nsoc-k} as $k \to \infty$, so $v_k \to v$. Now
the preliminary assumption from~\eqref{eq:prelim-assu} becomes important. Since
$v_k \to v$ in $\Le^r(0,\sT;\Le^2(\Omega))$, we know that the integrand $\xi_k(t)$ in
$\sj''(\bar u;v_k^2) = \int_0^\sT\xi_k(t)$ converges to the integrand in
$\sj''(\bar u;v^2)$ pointwise almost everywhere on $(0,\sT)$. It is moreover
nonnegative and bounded by
\begin{align*}
  \xi_k(t) &  \leq  \chi_{\setof{\|\bar u\|_{\Le^2(\Omega)} \neq 0}}(t)
  \frac{\|v_k(t)\|_{\Le^2(\Omega)}^2} {\|\bar
    u(t)\|_{\Le^2(\Omega)}} \\ & \leq  \chi_{\setof{\|\bar u\|_{\Le^2(\Omega)} \neq 0}}(t)
  \frac{\bigl(\|v(t)\|_{\Le^2(\Omega)}+\chi_{\setof{\ip{\bar u(t),v(t)} < 0} \,\cap \,\cA}(t)\,
    \cdot \,\omega(t)\bigr)^2} {\|\bar
    u(t)\|_{\Le^2(\Omega)}}
  \\ & \leq 2 \chi_{\setof{\|\bar u\|_{\Le^2(\Omega)} \neq 0}}(t) \left(
    \frac{\|v(t)\|_{\Le^2(\Omega)}^2 }{\|\bar
      u(t)\|_{\Le^2(\Omega)}} +  \omega(t)\right).
\end{align*}
By~\eqref{eq:prelim-assu}, the right-hand side is integrable over $(0,\sT)$. Thus we
can use the dominated convergence theorem to infer that
$\sj''(\bar u;v_k^2) \to \sj''(\bar u;v^2)$ as $k \to \infty$. We again argue analogously
for the multiplier term in~\eqref{eq:nsoc-k}. Since $F''(\bar u)$ is a
continuous bilinear form on $\Le^r(0,\sT;\Le^2(\Omega))$, we also have
$F''(\bar u)v_k^2 \to F''(\bar u)v^2$ as $k \to \infty$. Overall we obtain
\begin{multline*}
  0 \leq F''(\bar u)v^2 + \beta_1 \sj''(\bar u;v^2) \\ + \int_{\cA_+}
    \bar\mu(t) \|\bar u(t)\|_{\Le^2(\Omega)}^{-1} \left[\|v(t)\|_{\Le^2(\Omega)}^2 -
      \left(\frac{\bip{\bar u(t), v(t)}}{\| \bar
          u(t)\|_{\Le^2(\Omega)}}\right)^2\right] \dd t
\end{multline*}
for all $v \in C(\bar u) \cap \Le^\infty(0,\sT;\Le^2(\Omega))$
satisfying~\eqref{eq:prelim-assu}.

\textbf{Step 4}: \emph{Removing the additional assumptions}. To finally remove the assumptions on $v \in C(\bar u)$, we do another
approximation. Let $v\in C(\bar u)$ and let $\nu_\ell$ be a positive sequence
converging monotonically to zero. Define
\begin{equation*}
  N_\ell := \Bigl\{ t \in \bigl[\|\bar u\|_{\Le^2(\Omega)} \neq 0 \bigr] \colon 
  \|v(t)\|_{\Le^2(\Omega)} >
    \nu_\ell^{-1}\min(\|\bar u(t)\|_{\Le^2(\Omega)},1)\Bigr\}
  \end{equation*}
   and set $v_\ell \defn \chi_{N_\ell^c} v$.
Then for every $\ell$, we have $v_\ell \in \Le^\infty(0,\sT;\Le^2(\Omega))$ and
\begin{equation*}
  \int_{\setof{\|\bar u\|_{\Le^2(\Omega)} \neq 0}} \frac{\|v_\ell(t)\|_{\Le^2(\Omega)}^2}{\|\bar
    u(t)\|_{\Le^2(\Omega)}} \dd t = \int_{N_\ell^c}
  \frac{\|v(t)\|_{\Le^2(\Omega)}^2} {\|\bar
    u(t)\|_{\Le^2(\Omega)}} \dd t \leq   \nu_\ell^{-1} \int_{N_\ell^c}
  \|v(t)\|_{\Le^2(\Omega)} \dd t < \infty.
\end{equation*}
Moreover, by Lemma~\ref{lem:critical-direct-multiple-prop}, $v_\ell \in C(\bar u)$,
so by the above
\begin{multline*}
  0 \leq F''(\bar u)v_\ell^2 + \beta_1 \sj''(\bar u;v_\ell^2) \\ + \int_{\cA_+}
  \bar\mu(t) \|\bar u(t)\|_{\Le^2(\Omega)}^{-1} \left[\|v_\ell(t)\|_{\Le^2(\Omega)}^2
    - \left(\frac{\bip{\bar u(t), v_\ell(t)}}{\| \bar
        u(t)\|_{\Le^2(\Omega)}}\right)^2\right] \dd t.
\end{multline*}
We have $v_\ell \to v$ in $\Le^r(0,\sT;\Le^2(\Omega))$ as $\ell \to \infty$ due to
$|N_\ell| \to 0$.
So
$F''(\bar u)v_\ell^2 \to F''(\bar u)v^2$ as $\ell \to \infty$. Set moreover
\begin{equation*}
  \xi_\ell(t) := \chi_{N_\ell^c}(t)
  \, \cdot \,
  \|\bar u(t)\|_{\Le^2(\Omega)}^{-1}
  \left[\|v_\ell(t)\|_{\Le^2(\Omega)}^2 - \left(\frac{\bip{\bar
        u(t),v_\ell(t)}}{\|\bar u(t)\|_{\Le^2(\Omega)}}\right)^2\right] \dd t.
\end{equation*}
Then $0 \leq \xi_\ell(t) \leq \xi_{\ell+1}(t)$ for every $t \in (0,\sT)$ due to
$N_{\ell+1}\subseteq N_\ell$. Thus, the
monotone convergence theorem yields
\begin{equation*}
  \lim_{\ell\to\infty} \sj''(\bar u;v_\ell^2) =
  \lim_{\ell\to\infty}\int_0^\sT \xi_\ell(t) = \int_0^\sT
  \lim_{\ell\to\infty} \xi_\ell(t) = \sj''(\bar u;v^2).
\end{equation*}
Again, we argue analogously for the multiplier term and finally obtain
\begin{multline*}
  0 \leq F''(\bar u)v^2 + \beta_1 \sj''(\bar u;v^2) \\ + \int_{\cA_+}
  \bar\mu(t) \|\bar u(t)\|_{\Le^2(\Omega)}^{-1} \left[\|v(t)\|_{\Le^2(\Omega)}^2 -
    \left(\frac{\bip{\bar u(t), v(t)}}{\| \bar
        u(t)\|_{\Le^2(\Omega)}}\right)^2\right] \dd t
\end{multline*}
for \emph{all} $v \in C(\bar u)$. This was the claim.
\end{proof}

\subsubsection{Second order sufficient conditions}

We finally prove no-gap second order sufficient conditions for~\eqref{eq:ROCP} for
strong local solutions, i.e., in the $\Le^\infty(0,\sT;\Le^2(\Omega))$-sense. The
proof is fairly standard, but we again have to circumvent the possible singularity of
the substitute for the second derivative $\sj''(\bar u;v^2)$. We assume $\beta_2 > 0$
to enforce coercivity of the problem. The case $\beta_2 = 0$ is an open problem.

\begin{theorem}[Second order sufficient conditions]
  \label{thm:ssoc}
  Let $\beta_2 > 0$. Assume that $\bar u \in \uad$ satisfies
  \begin{equation}
    \label{eq:ssoc-condition}
    \ell_r''(\bar u;v^2) > 0 \quad \text{for all}~v \in C(\bar u)\setminus\{0\}.
  \end{equation}
  Then $\bar u$ is a strong local minimum, that is, there are $\eps,\delta > 0$ such
  that
  \begin{multline}
    \label{eq:quadratic-growth}
    \ell_r(\bar u) + \frac{\delta}2 \bigl\| u -
    \bar u\bigr\|_{\Le^2(0,\sT;\Le^2(\Omega))}^2 \leq \ell_r(u) \\ \text{for all}~u \in
    \uad~\text{with}~\bigl\|u - \bar u \bigr\|_{\Le^\infty(0,\sT;\Le^2(\Omega))} < \eps.
  \end{multline}
\end{theorem}

\begin{proof}
  Suppose not. Then there are positive nonincreasing sequences
  $\alpha_k,\eta_k \searrow 0$ and feasible controls $(u_k) \subset \uad$ such that
  \begin{equation}\label{eq:sosc-contra}
    \bigl\|u_k - \bar u\bigr\|_{\Le^\infty(0,\sT;\Le^2(\Omega))} \leq \alpha_k \quad
    \text{and} \quad \ell_r(u_k) < \ell_r(\bar u) + \frac{\eta_k}{2}\bigl\|u_k - \bar
    u\bigr\|_{\Le^2(0,\sT;\Le^2(\Omega))}^2. 
  \end{equation}
  We set $\rho_k \defn \|u_k - \bar u\|_{\Le^2(0,\sT;\Le^2(\Omega))}>0$ and
  $v_k \defn \rho_k^{-1}\bigl(u_k-\bar u\bigr)$. Then clearly $\rho_k \searrow
  0$. Since $v_k$ is normalized in $\Le^2(0,\sT;\Le^2(\Omega))$, we have
  $v_k \wto v \in \Le^2(0,\sT;\Le^2(\Omega))$, possibly after going over to a
  subsequence. The rest of the proof will consist of showing that $v = 0$ which then
  will lead to a contradiction with $\beta_2 > 0$. In order to show $v = 0$, we
  establish that $v \in C(\bar u)$ and $\ell_r''(\bar u;v^2) \leq 0$ and then
  conclude from~\eqref{eq:ssoc-condition}.

  \textbf{Step 1}: $v \in C(\bar u)$. We first show that $v \in \cT(\bar u)$. Since
  $u_k \in \uad$ for each $k$, we have
  $\|u_k(t)\|_{\Le^2(\Omega)}^2 \leq \omega(t)^2 = \|\bar u(t)\|_{\Le^2(\Omega)}^2$
  for almost all $t \in \cA$. Hence $\ip{u_k(t)-\bar u(t),\bar u(t)} \leq 0$ for
  $t \in \cA_+$ and $u_k(t) = \bar u(t) = 0$ for $t \in \cA_0$. This implies
  $v_k \in \cT(\bar u)$ and thus $v \in \cT(\bar u)$, since $\cT(\bar u)$ is weakly
  closed in $\Le^2(0,\sT;\Le^2(\Omega))$.

It remains to show that $F'(\bar u)v + \beta_1\sj'(\bar u;v) = 0$. We argue as
in~\cite[Proof~of Thm.~5.12]{CHW17}. The function $\sj$ is Lipschitz continuous and
convex, thus
\begin{equation*}
  \sj'(\bar u;v) \leq \liminf_{k \to \infty}\frac{\sj(u_k) - \sj(\bar u)}{\rho_k}.
\end{equation*}
This inequality and~\eqref{eq:sosc-contra} then show that
\begin{align*}
  F'(\bar u)v + \beta_1\sj'(\bar u;v) & \leq \liminf_{k \to \infty}\frac{\ell_r(u_k) -
    \ell_r(\bar u)}{\rho_k} \\ & \leq \liminf_{k\to\infty}\frac{\eta_k}{2\rho_k}
  \bigl\|u_k - \bar 
  u\bigr\|_{\Le^2(0,\sT;\Le^2(\Omega))}^2 = \liminf_{k\to\infty}
  \frac{\rho_k\eta_k}{2} = 0.
\end{align*}
From the reverse inequality from the first order necessary optimality
condition~\eqref{eq:FONC-direc-deriv} we infer $v \in C(\bar u)$.

\textbf{Step 2}: $\ell_r''(\bar u;v^2) \leq 0$: We again
define approximations to $v$ and $v_k$ to cope with the possible unboundedness of
$\sj''(\bar u;v^2)$. Let $\kappa_\ell > 0$ be a nonincreasing sequence converging to
zero. Set
\begin{equation*}
  v_{k,\ell}(t) \defn
  \begin{cases}
    0 & \text{if}~t \in \bsetof{0 < \|\bar u\|_{\Le^2(\Omega)} < \kappa_\ell}, \\
    v_k(t) & \text{elsewhere}
  \end{cases}
\end{equation*}
and $v_\ell$ analogously. Then $v_{k,\ell} \wto v_\ell$ for $\ell$ fixed and
$k \to \infty$, and $v_\ell \to v$ if $\ell \to \infty$, both in
$\Le^2(0,\sT;\Le^2(\Omega))$. We find for almost all $t \in (0,\sT)$
\begin{equation*}
  \bigl\|\bar u(t) + \rho_k v_{k,\ell}(t)\bigr\|_{\Le^2(\Omega)} \geq \kappa_\ell -
  \alpha_k \geq \frac{\kappa_\ell}2,
\end{equation*}
for $\ell$ fixed and $k$ large enough, since from~\eqref{eq:sosc-contra}
\begin{equation*}
  \rho_k \|v_{k,\ell}(t)\|_{\Le^2(\Omega)} \leq \rho_k \|v_k(t)\|_{\Le^2(\Omega)} \leq
  \alpha_k.
\end{equation*}
Hence we have integrated Taylor expansions of
$\|\bar u + \rho_k v_{k,\ell}\|_{\Le^2(\Omega)}$ from
Lemma~\ref{lem:taylor-expansion-norm}~\eqref{item:taylor-expansion-norm-integrated} at hand. Further, from
Lemma~\ref{lem:taylor-expansion-norm}~\eqref{item:taylor-expansion-norm-ineq},
\begin{multline}
  \int_{\setof{0 < \|\bar u\|_{\Le^2(\Omega)} < \kappa_\ell}} \Bigl(\bigl\|\bar u(t)
  + \rho_k v_{k}(t)\bigr\|_{\Le^2(\Omega)} - \bigl\|\bar u(t)\bigr\|_{\Le^2(\Omega)}
  \Bigr)\dd t \\ \geq \rho_k \int_{\setof{0 < \|\bar u\|_{\Le^2(\Omega)} < \kappa_\ell}}
  \frac{\bip{\bar u(t),v_k(t)}}{\|\bar u(t)\|_{\Le^2(\Omega)}} \dd t.\label{eq:taylor-ineq-ssoc}
\end{multline}
Using~\eqref{eq:sosc-contra} and the definition of $v_{k,\ell}$, we expand
\begin{align*}
  \frac{\eta_k\rho_k^2}{2} & > \ell_r(u_k) - \ell_r(\bar u) \\ &= F(u_k) - F(\bar
  u)  +
  \beta_1\bigl(\sj(u_k) - \sj(\bar u)\bigr) \\& = F(u_k) - F(\bar u) +
  \beta_1\rho_k\int_{\setof{\|\bar u\|_{\Le^2(\Omega)}=0}} \|v_k(t)\|_{\Le^2(\Omega)}
  \dd t \\ & \quad + \beta_1  \int_{\setof{0 < \|\bar u\|_{\Le^2(\Omega)} <
      \kappa_\ell}} \Bigl(\bigl\|\bar 
  u(t) + \rho_k v_{k}(t)\bigr\|_{\Le^2(\Omega)} - \bigl\|\bar
  u(t)\bigr\|_{\Le^2(\Omega)} \Bigr)\dd t \\ & \quad + \beta_1 \Bigl(\sj(\bar u +
  \rho_k v_{k,\ell}) - \int_{\setof{\|\bar u\|_{\Le^2(\Omega)} \geq
      \kappa_\ell}} \|\bar u(t)\|_{\Le^2(\Omega)} \dd t\Bigr).
\end{align*}
We insert Taylor expansions for $F$ and $\sj(\bar u + \rho_k v_{k,\ell})$
(cf.~Lemma~\ref{lem:taylor-expansion-norm}~\eqref{item:taylor-expansion-norm-integrated})
as well as~\eqref{eq:taylor-ineq-ssoc} to obtain, with a function
$\vartheta_k \colon [0,\sT] \to [0,1]$,
\begin{align*}
  \frac{\eta_k\rho_k}{2} & > \rho_k F'(\bar u)v_k
  + \frac{\rho_k^2}2 F''\bigl(\bar
  u+\vartheta_k\rho_kv_k\bigr)v_k^2 + 
  \beta_1\rho_k\int_{\setof{\|\bar u\|_{\Le^2(\Omega)}=0}} \|v_k(t)\|_{\Le^2(\Omega)} 
  \dd t  \\ & \quad + \beta_1\rho_k\left(\int_{\setof{0 < \|\bar
        u\|_{\Le^2(\Omega)} < \kappa_\ell}} \frac{\bip{\bar u(t),v_k(t)}}{\|\bar
      u(t)\|_{\Le^2(\Omega)}} \dd t  + \int_{\setof{\|\bar
        u\|_{\Le^2(\Omega)} \geq \kappa_\ell}} \frac{\bip{\bar u(t),v_k(t)}}{\|\bar
      u(t)\|_{\Le^2(\Omega)}} \dd t\right) \\ & \quad + \frac{\beta_1\rho_k^2}2
  \int_{\setof{\|\bar 
      u\|_{\Le^2(\Omega)} \geq \kappa_\ell}}  \|\bar
  u(t)\|_{\Le^2(\Omega)}^{-1} \left[ 
    \bigl\|v_{k,\ell}(t)\bigr\|_{\Le^2(\Omega)}^2 - \left(\frac{\bip{\bar u(t),
          v_{k,\ell}(t)}}{\| 
        \bar u(t)\|_{\Le^2(\Omega)}}\right)^2\right] \dd t \\ & \quad +
  \cO\bigl(\rho_k^3\|v_{k}\|_{\Le^3(0,\sT;\Le^2(\Omega))}^3\bigr).
\end{align*}
Note that
$\rho_k\|v_{k}\|_{\Le^3(0,\sT;\Le^2(\Omega))}^3 \leq
\rho_k\|v_k\|_{\Le^\infty(0,\sT;\Le^2(\Omega))} \in \cO(\alpha_k)$
by~\eqref{eq:sosc-contra}. Inserting this and rearranging, we arrive at
\begin{multline*}
  \frac{\eta_k\rho_k^2}{2} > \rho_k\bigl(F'(\bar u)v_k + \beta_1\sj'(\bar
  u;v_k)\bigr) \\ +\frac{\rho_k^2}{2} \left(F''\bigl(\bar u + \vartheta_k\rho_kv_k\bigr)v_k^2 +
    \beta_1\sj''(\bar u;v_{k,\ell}^2)\right) + 
  \cO\left(\alpha_k\rho_k^2\right).
\end{multline*}
Since $v_k \in \cT(\bar u)$, we have $F'(\bar u)v_k + \beta_1\sj'(\bar u;v_k) \geq 0$
by first order optimality~\eqref{eq:FONC-direc-deriv}. It is thus practical to divide
by $\rho_k^2/2$ and insert a zero to obtain
\begin{equation}
  \label{eq:sosc-nearly}
  F''(\bar u)v_k^2 + \beta_1\sj''(\bar u;v_{k,\ell}^2) \leq
  \bigl|F''(\bar u)v_k^2 - F''\bigl(\bar u +\vartheta_k\rho_kv_k\bigr)v_k^2\bigr| +
  \cO(\alpha_k) + \eta_k.
\end{equation}
We had $\rho_kv_k \to 0$ in $\Le^\infty(0,\sT;\Le^2(\Omega))$ as $k \to \infty$ by
construction. Hence, since $F$ is twice continuously differentiable by
Lemma~\ref{lem:f-c2} and $v_k$ is normalized in $\Le^2(0,\sT;\Le^2(\Omega))$,
\begin{equation*}
  \bigl|F''(\bar u)v_k^2 - F''\bigl(\bar u
  +\vartheta_k\rho_kv_k\bigr)v_k^2\bigr| \to 0 \quad \text{as}~k\to\infty.
\end{equation*}
It follows from Lemma~\ref{lem:f-c2} and
Lemma~\ref{lem:taylor-expansion-norm}~\eqref{item:taylor-expansion-norm-quadratic}
that the second derivatives on the left-hand side in~\eqref{eq:sosc-nearly} are both
weakly lower semicontinuous with respect to their directions. We thus obtain
from~\eqref{eq:sosc-nearly}
\begin{align*}
  F''(\bar u)v^2 + \beta_1\sj''(\bar u;v_\ell^2) 
  & \leq
  \liminf_{k\to\infty} \bigl( F''(\bar u)v_k^2 + 
  \beta_1\sj''(\bar u;v_{k,\ell}^2)\bigr) = 0.
\end{align*}
Finally, a monotone convergence argument as in the proof of Theorem~\ref{thm:sonc}
shows that 
\begin{equation*}
  \ell_r''(\bar u;v^2) =  F''(\bar u)v^2 + \beta_1 \lim_{\ell\to\infty}\sj''(\bar u;v_\ell^2)
  \leq 0
\end{equation*}
so that $v = 0$ by~\eqref{eq:ssoc-condition}.

\textbf{Step 3}: Contradiction to $\beta_2 > 0$. From
$v_k \wto v = 0$, it follows that $z_{v_k} \wto z_v = 0$ in
$\cY$ due to Theorem~\ref{thm:control-to-state-diff}. Hence,
setting $G(u) \defn F(u) - \frac{\beta_2}2
\|u\|_{\Le^2(0,\sT;\Le^2(\Omega))}^2$ and using
Lemma~\ref{lem:f-c2} for $\beta_2 = 0$, we find
\begin{equation*}
  0 = G''(\bar u)z_v^2 \leq \liminf_{k\to\infty} G''(\bar u)z_{v_k}^2.
\end{equation*}
Since $v_k$ was normalized in $\Le^2(0,\sT;\Le^2(\Omega))$ and
$\sj''(\bar u;v_{k,\ell}^2) \geq 0$ for every $\ell$, we thus
have
\begin{align*}
  \beta_2 & \leq \liminf_{k\to\infty}G''(\bar
  u)z_{v_k}^2 + \beta_2  \\ & = \liminf_{k\to\infty}G''(\bar
  u)z_{v_k}^2 + \beta_2\liminf_{k\to\infty} \|v_k\|_{\Le^2(0,\sT;\Le^2(\Omega))} \\ & \leq \liminf_{k\to\infty}
  F''(\bar u)v_k^2\\ & \leq \liminf_{k\to\infty}
  F''(\bar u)v_k^2 + \beta_1\liminf_{k\to\infty} \sj''(\bar
  u;v_{k,\ell}^2)  \\ & \leq \liminf_{k\to\infty} \left(
  F''(\bar u)v_k^2 + \beta_1 \sj''(\bar
  u;v_{k,\ell}^2)\right) = 0,
\end{align*}
where again~\eqref{eq:sosc-nearly} was used. 
This is the final contradiction and completes the proof.
\end{proof}

\appendix
\def\appendixname{Appendix}
\section{Auxiliary results}
\renewcommand*{\thesection}{\Alph{section}}

At several points we need the following growth lemma.

\begin{lemma}[{\cite[Ch.~IV, Lem.~2.2]{S08}}]
  \label{lem:useful-estimate} Let $C_0 > 0$ and suppose that $f \in \C([a,b))$
  satisfies $f(s) \geq 0$ for all $s \in [a,b)$ and $f(a) = 0$ as well as
  \begin{equation*}
    f(s) \leq C_0 + \eps f(s)^\sigma  \quad \text{for all}~s\in [a,b)
  \end{equation*}
  for some $\sigma > 0$. If $\eps < 2^{-\sigma}C_0^{1-\sigma}$, then it follows that
  \begin{equation*}
    f(s) \leq 2C_0 \quad \text{for all}~s\in [a,b).
  \end{equation*}
\end{lemma}

The remaining assertions are about the first and second derivatives of (powers
of) norms on Banach spaces. We need the results for Lebesgue spaces
for which the following result is a basic one.

\begin{proposition}[{\cite[Thms.~3.3\&3.9]{LS74}}]
  Let $\Upsilon \subset \R^d$ for some $d \in \N$ and let $X$ be a Banach space. Then
  the norm on $\Le^p(\Upsilon;X)$ for $2 < p < \infty$ is twice continuously
  differentiable away from $0$ if and only if the norm on $X$ is twice continuously
  differentiable away from $0$ and the second derivative is uniformly bounded on the
  unit sphere in $X$. The norm on $\Le^2(\Upsilon;X)$ is twice continuously
  differentiable away from $0$ if and only if $X$ is a Hilbert
  space.\label{prop:norm-differentiability}
\end{proposition}

The requirements in Proposition~\ref{prop:norm-differentiability} are clearly
satisfied in the case $X = \R$. Setting $\Upsilon_p(f) \defn \|f\|_{\Le^p(\Omega)}$ for
$2 \leq p < \infty$, we have the following derivatives in
$f \in \Le^p(\Omega)\setminus \{0\}$:
\begin{align*}
  \Upsilon_p'(f)h & = \|f\|^{1-p}_{\Le^p(\Omega)} \bip{|f|^{p-2}f,h} ,\\
   \Upsilon_p''(f)h_1h_2 & = (p-1) \left(
   \|f\|_{\Le^p(\Omega)}^{1-p}\bip{|f|^{p-2}h_1,h_2} -
   \|f\|_{\Le^p(\Omega)}^{-1} \prod_{i=1}^2 \bigl(\Upsilon_p'(f)h_i\bigr)\right)
\end{align*}
The H\"older inequality shows that
$|\Upsilon_p''(f)h_1h_2|\lesssim
\|f\|_{\Le^p(\Omega)}^{-1}\|h_1\|_{\Le^p(\Omega)}\|h_2\|_{\Le^p(\Omega)}$,
so $\Upsilon_p''$ is bounded on the unit sphere in
$\Le^p(\Omega)$. This allows to invoke
Proposition~\ref{prop:norm-differentiability} again for
$\Phi_{p,q}(f) \defn \|f\|_{\Le^p(0,\sT;\Le^q(\Omega))}$ with
$2 < p,q < \infty$ or $p=q=2$, and a straight forward calculation with
the chain rule gives the derivatives in $f \neq 0$ as follows:
\begin{align*}
  \Phi_{p,q}'(f)h & = \|f\|^{1-p}_{\Le^p(0,\sT;\Le^q(\Omega))} \int_0^\sT
  \|f(t)\|_{\Le^q(\Omega)}^{p-1} \Upsilon_q'(f(t))h(t) \dd t ,  \\ 
  \Phi_{p,q}''(f)h_1h_2 & = (p-1)\|f\|^{1-p}_{\Le^p(0,\sT;\Le^q(\Omega))} \int_0^\sT
   \|f(t)\|_{\Le^q(\Omega)}^{p-2}
   \prod_{i=1}^2\Bigl(\Upsilon_q'(f(t))h_i(t)\Bigr) \dd t  \\ & \quad -
  (p-1)\|f\|^{1-2p}_{\Le^p(0,\sT;\Le^q(\Omega))} \prod_{i=1}^2 \int_0^\sT
  \|f(t)\|_{\Le^q(\Omega)}^{p-1} \Upsilon_q'(f(t))h_i(t) \dd t \\ & \qquad +
  \|f\|^{1-p}_{\Le^p(0,\sT;\Le^q(\Omega))} \int_0^\sT \|f(t)\|_{\Le^q(\Omega)}^{p-1}
  \Upsilon_q''(f(t))h_1(t)h_2(t) \dd t.
\end{align*}

We next consider powers of the norms $\Phi_{p,q}$. The general result
is as follows. It shows that with a sufficiently large power one can
overcome the nondifferentiability in $0$ of $\Phi_{p,q}$.

\begin{lemma}
  \label{lem:lipschitz-diff}
  Let $X$ be a Banach space and $\nu > 1$. Consider 
  $f \colon X \to \R$ and set $g(x) \defn |f(x)|^{\nu-1}f(x)$.
  \begin{enumerate}[leftmargin=*]
  \item Suppose that $f$ is locally Lipschitz continuous in $0$ and continuously
    differentiable on $X \setminus \{0\}$. Then $g$ is continuously differentiable
    on $X$ with $g'(0) = 0$.
  \item Suppose in addition that $\nu > 2$ and that $f$ is twice
    continuously differentiable on $X \setminus \{0\}$ with
    $\|f''(x)\|_{\cL(X\times X;\R)} \lesssim \|x\|_X^{-1}$ as $x
    \to 0$. Then
    $g$ is twice continuously differentiable on $X$ with $g''(0)
    = 0$.
  \end{enumerate}
\end{lemma}

\begin{proof} We can without loss of generality assume that
  $f(0) = 0$.
  \begin{enumerate}[wide, labelwidth=!, labelindent=0pt]
  \item It is clear that $g$ is continuously differentiable on
    $X \setminus \{0\}$ with $g'(x)h =
    \nu|f(x)|^{\nu-1}f'(x)h$. Moreover, let $R>0$ be such that
    $f$ is Lipschitz continuous on a ball around $0$ with radius
    $R$. We denote the Lipschitz constant by $L_R$. Then we have
    for all $h$ with $\|h\|_X \leq R$
    \begin{equation*}
      \bigl|g(h)-g(0) - 0\cdot h\bigr| = |f(h)|^\nu \leq L_R^\nu
      \|h\|_X^\nu \in o\bigl(\|h\|_X\bigr),
    \end{equation*}
    so $g$ is differentiable in $0$ with derivative $g'(0) =
    0$. Since $\|f'(x)\|_{X'} \leq 1$ for all
    $X\setminus \{0\}$, we further find
    \begin{equation*}
      |g'(x)h| \leq \nu L_R^{\nu-1}\|x\|_X^{\nu-1}\|h\|_X
    \end{equation*}
    for all $x$ with $\|x\|_X \leq R$. This shows that
    $\|g'(x)\|_{X'} \to 0$ as $x\to0$ in $X$, so $g'$ is
    continuous in $0$.
  \item Now $g$ is even twice continuously differentiable on
    $X \setminus \{0\}$ with
    \begin{equation*}
      g''(x)h_1h_2 =
      \nu(\nu-1)|f(x)|^{\nu-2}\bigl(f'(x)h_1\bigr)\bigl(f'(x)h_2\bigr)
      + \nu |f(x)|^{\nu-1}f''(x)h_1h_2.
    \end{equation*}
    As above, we obtain, for all $h_2 \in X$ and $h_1$ with
    $\bigl\|h_1\bigr\|_X \leq R$,
    \begin{equation*}
      \bigl|g'(h_1)h_2 - g'(0)h_2 - 0\cdot h_1h_2\bigr| \leq \nu
      L_R^{\nu-1}\bigl\|h_1\bigr\|_X^{\nu-1}\bigl\|h_2\bigr\|_X,
    \end{equation*}
    hence $\|g'(h_1) - g'(0) - 0\|_{X'} \in o\bigl(\|h_1\|_X\bigr)$ and
    $g'$ is  differentiable in $0$ with $g''(0) = 0$. Here we have used that $\nu >
    2$ now. The claim
    that $g''$ is continuous in $0$
    follows from the observation that
    \begin{equation*}
      \bigl|g''(x)h_1h_2\bigr|
      \lesssim \|x\|^{\nu-2}\bigl\|h_1\bigr\|_X\bigl\|h_2\bigr\|_X
    \end{equation*}
    for all $h_1,h_2 \in X \setminus \{0\}$ as $x\to 0$ by the
    assumption on $f''$, and again $\nu > 2$. \qedhere
  \end{enumerate}
\end{proof}

We set
$\Psi_{p,q}(y) \defn \frac1p \|y\|_{\Le^p(0,\sT;\Le^{q}(\Omega))}^p =
\frac1p \Phi_{p,q}(y)^p$ and want to show that it is twice
continuously differentiable for $p > 2$ using
Lemma~\ref{lem:lipschitz-diff} with $f = \Phi_{p,q}$ and $\nu =
p$. The Lipschitz condition in Lemma~\ref{lem:lipschitz-diff} is
clearly satisfied since $\Phi_{p,q}$ is a norm, and we had already
seen that Proposition~\ref{prop:norm-differentiability} implies the
differentiability assumptions. Moreover, repeated use of H\"older's
inequality in the foregoing formula for $\Phi''_{p,q}$ shows that
  \begin{equation*}
    \bigl|\Phi_{p,q}''(f)h_1h_2\bigr| \lesssim
    \|f\|_{\Le^p(0,\sT;\Le^{q}(\Omega))}^{-1}
    \|h_1\|_{\Le^p(0,\sT;\Le^{q}(\Omega))}\|h_2\|_{\Le^p(0,\sT;\Le^{q}(\Omega))}. 
  \end{equation*}
Hence we can indeed apply Lemma~\ref{lem:lipschitz-diff} with
$f = \Phi_{p,q}$ and $\nu = p$ to obtain the following:

\begin{corollary}
  \label{cor:psi-diff}
  For $p > 2,$ the function $\Psi_{p,q}$ is twice continuously
  differentiable as a mapping from $\Le^p(0,\sT;\Le^q(\Omega))$ to
  $\R$. Its derivatives are given by
\begin{multline*}
  \Psi_{p,q}'(y)h = \int_0^\sT \|y(t)\|_{\Le^{q}(\Omega)}^{p-1}
    \Upsilon_q'(y(t))h(t) \dd t \\ = \int_0^\sT \|y(t)\|_{\Le^{q}(\Omega)}^{p-q}
    \bip{|y(t)|^{q-2}y(t),h(t)} \dd t
\end{multline*}
and
\begin{multline*}
  \Psi_{p,q}''(y)h^2 = (p-q)\int_0^\sT \|y(t)\|_{\Le^q(\Omega)}^{p-2}
  \bigl(\Upsilon_q'(y(t))h(t)\bigr)^2 \dd t \\ + (q-1)\int_0^\sT
  \|y(t)\|_{\Le^q(\Omega)}^{p-q} \bip{|y(t)|^{q-2},h^2(t} \dd t.
\end{multline*}
\end{corollary}

\section*{Acknowledgements}

The work of Karl Kunisch was partly supported by the ERC
advanced grant~668998 (OCLOC) under the EU~H2020 research
program.



\bibliographystyle{siam}
\bibliography{Wave}

\end{document}